\documentclass[11pt]{amsart}

\usepackage{amssymb}
\usepackage{graphicx}   
\usepackage{mathrsfs}
\usepackage{amsmath}
\usepackage{amsthm}
\usepackage{amsfonts}
\usepackage{latexsym}

\newtheorem{thm}{Theorem}[section]
\newtheorem{prop}[thm]{Proposition}
\newtheorem{lem}[thm]{Lemma}
\newtheorem{cor}[thm]{Corollary}

\theoremstyle{definition}

\theoremstyle{claim}

\theoremstyle{remark}
\newtheorem{remark}[thm]{Remark}

\numberwithin{equation}{section}

\DeclareMathOperator{\dist}{dist} 

\begin{document}


\title{Mountain Pass Energies Between Homotopy Classes of Maps}

\author{Daniel Stern}
\address{Department of Mathematics, Princeton University, 
Princeton, NJ 08544}
\email{dls6@math.princeton.edu}

\begin{abstract} For non-homotopic maps $u,v\in C^{\infty}(M,N)$ between closed Riemannian manifolds, we consider the smallest energy level $\gamma_p(u,v)$ for which there exist paths $u_t\in W^{1,p}(M,N)$ connecting $u_0=u$ to $u_1=v$ with $\|du_t\|_{L^p}^p\leq \gamma_p(u,v)$. When $u$ and $v$ are $(k-2)$-homotopic, work of Hang and Lin shows that $\gamma_p(u,v)<\infty$ for $p\in [1,k)$, and using their construction, one can obtain an estimate of the form $\gamma_p(u,v)\leq \frac{C(u,v)}{k-p}$. When $M$ and $N$ are oriented, and $u$ and $v$ induce different maps on real cohomology in degree $k-1$, we show that the growth $\gamma_p(u,v)\sim \frac{1}{k-p}$ as $p\to k$ is sharp, and obtain a lower bound for the coefficient $\liminf_{p\to k}(k-p)\gamma_p(u,v)$ in terms of the min-max masses of certain non-contractible loops in the space of codimension-$k$ integral cycles in $M$. In the process, we establish lower bounds for a related smaller quantity $\gamma_p^*(u,v)\leq\gamma_p(u,v)$, for which there exist critical points $u_p\in W^{1,p}(M,N)$ of the $p$-energy functional satisfying $\gamma_p^*(u,v)\leq \|du_p\|_{L^p}^p\leq \gamma_p(u,v).$

\end{abstract}

\maketitle



\section{Introduction}

\hspace{3mm} Let $M^n$ and $N$ be closed Riemannian manifolds, with $N\subset \mathbb{R}^L$ isometrically embedded in some higher-dimensional Euclidean space. For $p\geq 1$, the space $W^{1,p}(M,N)$ of Sobolev maps from $M$ to $N$ is defined by
$$W^{1,p}(M,N):=\{u\in W^{1,p}(M,\mathbb{R}^L)\mid u(x)\in N\text{ for almost every }x\in M\}.$$

\hspace{3mm} In \cite{HangLin}, Hang and Lin characterized the path components of $W^{1,p}(M,N)$ in the low regularity setting $p\in [1,n)$, showing that two maps $u$ and $v$ lie in the same component of $W^{1,p}(M,N)$ if and only if they are $\lfloor p-1\rfloor$-homotopic--i.e., if their restrictions to the $\lfloor p-1\rfloor$-skeleton of a generic triangulation of $M$ are homotopic in the usual sense. The results of \cite{HangLin} build on earlier work of White, who showed in \cite{Wh2} that maps in $W^{1,p}(M,N)$ have well-defined $\lfloor p-1\rfloor$-homotopy classes that are closed under bounded weak convergence, and deduced that each $\lfloor p-1\rfloor$-homotopy class contains a representative minimizing the $p$-energy
$$E_p(u):=\int_M|du|^p.$$

\hspace{3mm} As a corollary of the results in \cite{HangLin}, Hang and Lin confirm a conjecture of Brezis and Li (\cite{BL}, Conjecture 2), which posits that, as $p$ varies in $[1,n)$, the spaces $W^{1,p}(M,N)$ undergo a ``change of topology"--i.e., a change in the number of path components--only when $p$ passes through integer thresholds. The results of the present paper are motivated by a desire to obtain a more quantitative understanding of this ``change of topology" phenomenon, with the aim of gaining new insight into how the topology of $M$ and $N$ influences the variational landscape of the $p$-energy functionals.

\hspace{3mm} Fix two maps $u,v\in C^{\infty}(M,N)$, and let $2\leq k\leq n=\dim M$ be the largest integer such that $u$ and $v$ are homotopic on the $(k-2)$-skeleton of some--hence, any (see \cite{HangLin}, Section 2.2)--triangulation of $M$. By the results of \cite{HangLin} and \cite{Wh2}, we then see that $u$ and $v$ lie in a common path component of $W^{1,p}(M,N)$ if and only if $p<k$. For $p$ close to $k$, we would like to characterize those energy levels $c>0$ for which the maps $u$ and $v$ lie in a common path component of the energy sublevel set
$$E_p^c:=\{w\in W^{1,p}(M,N)\mid \|dw\|_{L^p}^p\leq c\}.$$
That is, we are interested in estimating the mountain-pass energies 
\begin{equation}
\gamma_p(u,v):=\inf\{c>0\mid \exists\text{ a path }u_t\in E_p^c\text{ connecting }u\text{ to }v\};
\end{equation}
in the limit $p\to k$ where the change of topology occurs.

\hspace{3mm} First, we observe that a careful examination of the path constructed by Hang and Lin in the proof of (\cite{HangLin}, Theorem 1.1) yields the following upper bound. (Since these estimates are not explicitly addressed in \cite{HangLin}, we provide a proof in Section \ref{upperbds} of the appendix.)
\begin{thm}\label{ubdsthm} There exists a constant $C=C(u,v)<\infty$ such that
\begin{equation}
\gamma_p(u,v)\leq\frac{C}{k-p}
\end{equation}
for every $p\in [1,k)$.
\end{thm}

\hspace{3mm} Our main result shows that, when $M$ and $N$ are oriented, and $u$ and $v$ induce different maps on the real cohomology $u^*,v^*: H^{k-1}(N;\mathbb{R})\to H^{k-1}(M;\mathbb{R})$, the growth $\gamma_p(u,v)\sim \frac{1}{k-p}$ is in fact optimal, with an explicit lower bound on the coefficient 
$$\liminf_{p\to k}(k-p)\gamma_p(u,v).$$
In fact, we establish a lower bound for a possibly smaller quantity $\gamma_p^*(u,v)\leq \gamma_p(u,v)$, defined roughly as the smallest energy level $c>0$ for which $u$ and $v$ can be connected by sequences $u=u_0,u_1,\ldots,u_{r-1},u_r=v$ in the sub-level set $E_p^c$ for which adjacent maps $u_i,u_{i+1}$ are arbitrarily close in $L^p$ norm. (See Section \ref{lbdspf} below for a careful definition.)

\hspace{3mm} To state the lower bound precisely, we briefly introduce some relevant notation (to be defined in greater detail in Sections \ref{prelims}-\ref{lbdssec}). First, we denote by $\mathcal{A}^{k-1}(N)$ the space of closed $(k-1)$-forms on $N$ with the property that
$$\langle \alpha,\Sigma\rangle\in \mathbb{Z}$$
for every integral $(k-1)$-cycle $\Sigma$ in $N$, and for $\alpha\in \mathcal{A}^{k-1}(N)$, we use $S_{\alpha}(v)-S_{\alpha}(u)$ to denote the dual $(n+1-k)$-current associated to $v^*(\alpha)-u^*(\alpha)$ by
$$\langle S_{\alpha}(v)-S_{\alpha}(u),\zeta\rangle:=\int_M(v^*(\alpha)-u^*(\alpha))\wedge \zeta\text{\hspace{2mm} for }\zeta\in \Omega^{n+1-k}(M).$$

\hspace{3mm} Next, we recall from Almgren's dissertation \cite{A1} that there exists an isomorphism
$$\Phi: \pi_1(\mathcal{Z}_m(M;\mathbb{Z}),\{0\})\to H_{m+1}(M;\mathbb{Z})$$
relating loops in the space of integral $m$-cycles (with the flat topology) to integral $(m+1)$-homology classes in $M$. In Section \ref{almsubsec} below, we define for each $\xi\in H_{m+1}(M;\mathbb{Z})$ a min-max width ${\bf L}_m(\xi)>0$, which corresponds roughly to the min-max mass
$$\inf\{\sup_{t\in S^1}\mathbb{M}(\gamma(t))\mid \gamma: S^1\to \mathcal{Z}_m(M;\mathbb{Z}),\text{ }\Phi(\gamma)=\xi\}$$
associated to the class $\Phi^{-1}(\xi)\in \pi_1(\mathcal{Z}_m(M;\mathbb{Z}),\{0\})$. For any \emph{real} homology class $\overline{\xi}\in H_{m+1}(M;\mathbb{R})$ that can be represented by integral cycles, we then define 
$${\bf L}_{m,\mathbb{R}}(\overline{\xi}):=\min\{{\bf L}_m(\xi)\mid \xi \in H_{m+1}(M;\mathbb{Z}),\text{ }\xi\equiv \overline{\xi}\text{ in }H_{m+1}(M;\mathbb{R})\}.$$
Our main theorem then reads as follows.

\begin{thm}\label{lbdsthm} For any $\alpha\in \mathcal{A}^{k-1}(N)$ and maps $u,v\in C^{\infty}(M,N)$ such that 
$$[u^*(\alpha)-v^*(\alpha)]\neq 0\in H^{k-1}_{dR}(M),$$
there is a constant $\lambda(\alpha)<\infty$ such that
\begin{equation}
\lambda(\alpha)^{\frac{k}{k-1}}\liminf_{p\to k}(k-p)\gamma_p^*(u,v)\geq \sigma_{k-1}{\bf L}_{n-k,\mathbb{R}}([S_{\alpha}(v)-S_{\alpha}(u)]).
\end{equation}
\end{thm}

\begin{remark} Here, $\sigma_{k-1}$ denotes the $(k-1)$-volume of the standard unit $(k-1)$-sphere. The definition of $\lambda(\alpha)$ is given in Section \ref{degests}; for now we only remark that $\lambda(\alpha)$ is easy to estimate for specific choices of target $N$ and $\alpha\in \mathcal{A}^{k-1}(N)$. When $N=S^{k-1}$ is the standard $(k-1)$-sphere and $\alpha=\frac{dvol}{\sigma_{k-1}}$, for example, one has $\lambda(\alpha)=(k-1)^{\frac{1-k}{2}}$.
\end{remark}

\hspace{3mm} Though the details of the proof are somewhat delicate, the intuition underlying Theorem \ref{lbdsthm} is relatively straightforward. For any map $w\in W^{1,p}(M,N)$, $p\in (k-1,k)$, the pullback $w^*(\alpha)$ is well-defined as a $(k-1)$-form with coefficients in $L^1$; as in (\cite{GMS2}, Section 5.4.2), we can then define an $(n-k)$-current $T_{\alpha}(w)$ corresponding to the distributional exterior derivative of $w^*(\alpha)$. In Section \ref{limthms}, we develop a compactness theory for the so-called homological singularities $T_{\alpha}(w_p)$ for families of maps with $(k-p)E_p(w_p)\leq \Lambda$ as $p\to k$ (based largely on ideas from the $\Gamma$-convergence results of \cite{ABOvar} and \cite{JSgl} for functionals of Ginzburg-Landau type), showing that the currents $T_{\alpha}(w_p)$ converge subsequentially in $(C^1)^*$ to an integral $(n-k)$-cycle, whose mass we can bound explicitly in terms of the limiting energy $\liminf_{p\to k}(k-p)E_p(w_p)$. (In fact, a much more careful description of the convergence of $T_{\alpha}(w_p)$ is necessary for our applications.)

\hspace{3mm} To an $L^p$-fine sequence $u=u_0,u_1,\ldots,u_r=v$ of maps with $p$ close to $k$ and energy bounded above by $\gamma_p^*(u,v)+\epsilon$, we can then associate a family of integral $(n-k)$-cycles $0=T_0,T_1,\ldots,T_r=0$, with mass bounded in terms of $(k-p)\gamma_p^*(u,v)$. Using results of Sections \ref{homsingbasics} and \ref{limthms}, and some additional technical lemmas, we then show that the difference of adjacent cycles $T_{i-1},T_i$ in this family can be written $T_{i}-T_{i-1}=\partial S_i$ for integral $(n+1-k)$-currents $S_i$ of small mass, such that
$$[\Sigma_{i=1}^rS_i]\equiv S_{\alpha}(v)-S_{\alpha}(u)\text{ in }H_{n+1-k}(M;\mathbb{R}).$$
The conclusion of Theorem \ref{lbdsthm} follows from these observations.

\hspace{3mm} We suspect that the lower bound
\begin{equation}\label{maybebd}
\liminf_{p\to k}(k-p)\gamma_p^*(u,v)>0
\end{equation}
holds for non-$(k-1)$-homotopic maps $u,v\in C^{\infty}(M,N)$ in much greater generality. If $N$ is simply connected and $k>2$, for example, one might try to show this by associating to each path $u_t\in W^{1,p}(M,N)$ from $u$ to $v$ a loop of ``topological singularities" $T(u_t)$ given by flat $(n-k)$-cycles with coefficients in $\pi_{k-1}(N)$, with limiting behavior as $p\to k$ similar to that described in Section \ref{limthms} for the homological singularities considered here. The non-contractibility of the resulting loops in $\mathcal{Z}_{n-k}(M;\pi_{k-1}(N))$ may then be related to the nontriviality of classical cohomological obstructions to extending a given $(k-2)$-homotopy to a $(k-1)$-homotopy (see, e.g., \cite{Hu}, Sections VI.4 and VI.8). If the target $N$ is $(k-2)$-connected--and $\pi_1(N)$ is abelian, if $k=2$--the analysis of topological singularities pursued by Pakzad-Rivi\`{e}re \cite{PR} and Canevari-Orlandi \cite{CO} may provide a useful starting point for investigations in this direction. For general $N$ and $k$, however, addressing these questions will require the introduction of some new machinery.

\begin{remark}\label{s1case} In the case $N=S^1$, we observe that the general statement (\ref{maybebd}) follows immediately from Theorem \ref{lbdsthm}, since the homotopy classes $[u]\in [M:S^1]$ are determined by the pullback $[u^*(\alpha)]\in H^1_{dR}(M)$ of the generator $\alpha=\frac{1}{2\pi}d\theta\in \mathcal{A}^1(S^1)$ of $H^1_{dR}(S^1)$. In this case, Theorem \ref{lbdsthm} is closely related to questions raised by the author in \cite{St17} concerning the mountain-pass energies for complex Ginzburg-Landau functionals $E_{\epsilon}:W^{1,2}(M,\mathbb{C})\to \mathbb{R}$ over paths in $W^{1,2}(M,\mathbb{C})$ connecting distinct classes in $[M:S^1]$. On a two-dimensional annulus $\Omega$, these questions had previously been studied by Almeida \cite{Alme1},\cite{Alme2}, who obtained precise estimates for the $E_{\epsilon}$-mountain pass energies separating the components of $W^{1,2}(\Omega,S^1)$ in $W^{1,2}(\Omega,\mathbb{C})$. As discussed in \cite{Alme2}, these results have some physical relevance, since the presence of high energy walls separating local minimizers of Ginzburg-Landau energies on $\Omega$ is related to the appearance of permanent currents in superconducting cylinders.
\end{remark}

\hspace{3mm} Finally, in Section \ref{pharms}, we apply standard mountain pass arguments to the generalized Ginzburg-Landau functionals $E_{p,\epsilon}$ studied by Wang in \cite{Wang} to demonstrate the existence of critical points for the $p$-energy functional with energy lying between $\gamma_p^*(u,v)$ and $\gamma_p(u,v)$. In particular, we obtain the following result.

\begin{thm}\label{pharmexist}
If $u,v\in C^{\infty}(M,N)$ are $(k-2)$-homotopic, $p\in (k-1,k)$, and
$$\gamma_p^*(u,v)>\max\{E_p(u),E_p(v)\},$$
then there exists a stationary $p$-harmonic map $w\in W^{1,p}(M,N)$ with
$$\gamma_p^*(u,v)\leq E_p(w)\leq \gamma_p(u,v).$$
\end{thm}

\hspace{3mm} As a consequence, whenever (\ref{maybebd}) holds--as it does if $u$ and $v$ induce different maps on real cohomology, by Theorem \ref{lbdsthm}--we obtain for some $q=q(u,v)\in (k-1,k)$ a family of stationary $p$-harmonic maps 
$$(q,k)\ni p\mapsto u_p\in W^{1,p}(M,N)$$
whose energy grows like $\frac{1}{k-p}$ as $p\to k$. In the case $N=S^1$, the asymptotic behavior as $p\to 2$ of stationary $p$-harmonic maps with $E_p(u_p)=O(\frac{1}{2-p})$ was studied by the author in \cite{St18}, with the conclusion that the singular sets $Sing(u_p)$ converge subsequentially--in the Hausdorff sense--to the support of a stationary, rectifiable $(n-2)$-varifold in $M$. Though we expect the situation for general $N$ and $k$ to be considerably more complicated, it would likewise be interesting to understand the asymptotic behavior of stationary $p$-harmonic maps $u_p\in W^{1,p}(M,N)$ with $E_p(u_p)=O(\frac{1}{k-p})$ as $p\to k$.

\section*{Acknowledgements} I am grateful to Fernando Cod\'{a} Marques for his encouragement and constant support. During the completion of this project, the author was partially supported by NSF grants DMS-1502424 and DMS-1509027.

\section{Preliminaries}\label{prelims}

\subsection{Cubeulations, Slicing, and Retraction to Skeleta}\label{slices}\hspace{30mm}

\hspace{3mm} Throughout the paper, we will make repeated use of the fact that any compact Riemannian manifold $M$ admits a cubeulation (see, for instance, \cite{Wh1}). That is, we can find an $n$-dimensional cubical complex $K$ whose $n$-faces are isometric to $[-1,1]^n$, and a bi-Lipschitz map
$h: |K|\to M$ from the underlying space $|K|$ onto $M$. A key tool in the study of Sobolev maps (see, e.g., \cite{Wh1}, \cite{Wh2}, \cite{Bet}, \cite{HangLin}) is the ability to ``slice" a given Sobolev function $f$ by the skeleta of a well-chosen cubeulation $h:|K|\to M$, in such a way that the composition $f\circ h$ is well-behaved on every lower-dimensional skeleton $|K^j|$ of $K$. 

\hspace{3mm} To make this idea precise, we review some notation from \cite{HangLin}. For an $n$-dimensional cubical complex $K$ and a Lipschitz map $f\in Lip(|K|,\mathbb{R})$, define the $\mathcal{W}^{1,p}(K)$ norm of $f$ by
\begin{equation}
\|f\|_{\mathcal{W}^{1,p}(K)}^p:=\Sigma_{\sigma\in K}\int_{\sigma}(|f|^p+|df|^p)d\mathcal{H}^{\dim(\sigma)},
\end{equation}
and denote by $\mathcal{W}^{1,p}(K,\mathbb{R})$ the completion of  $Lip(|K|,\mathbb{R})$ with respect to this norm. Note that the functions in $\mathcal{W}^{1,p}(K,\mathbb{R})$ then lie in the usual Sobolev space $W^{1,p}(\Delta,\mathbb{R})$ for every cell $\Delta\in K$ of any dimension. Following arguments from Section 3 of \cite{HangLin} (see also the proof of Lemma 2.2 in \cite{Hang}), we obtain the following slicing lemma for Sobolev functions. 

\begin{lem}\label{slicelemma} Given a closed Riemannian manifold $M$, there exists a constant $C(M)<\infty$ and for each $\delta\in (0,1]$, there exists a cubical complex $K_{\delta}$ whose $n$-cells are isometric to $[-\delta,\delta]^n$, such that for any $f\in W^{1,p}(M,\mathbb{R})$, we can find a cubeulation $h:|K_{\delta}|\to M$ for which $f\circ h\in \mathcal{W}^{1,p}(K_{\delta},\mathbb{R})$, satisfying
\begin{equation}
Lip(h)+Lip(h^{-1})\leq C,
\end{equation}
\begin{equation}
\int_{|K_{\delta}^j|}|f\circ h|^pd\mathcal{H}^j\leq C\delta^{j-n}\int_M|f|^p,
\end{equation}
and
\begin{equation}
\int_{|K_{\delta}^j|}|d(f\circ h)|^pd\mathcal{H}^j\leq C\delta^{j-n}\int_M|df|^p
\end{equation}
for every $0\leq j\leq n$.
\end{lem}

\begin{remark}\label{combrk} In Section \ref{slicelemapp} of the appendix, we indicate how the methods of Sections 3 and 4 of \cite{HangLin} can be used to find such cubeulations. It is useful to note that the complexes $K_{\delta}$ are obtained (up to minor rescalings) by subdividing the faces of an initial unit-size complex $K=K_1$: in particular, it follows that the maximum number $\nu_j(K_{\delta})$ of $j$-cells containing a given $(j-1)$-cell as a face is fixed independent of $\delta$. As a consequence, whenever we have an estimate of the form $\int_{\sigma}f_1\leq \int_{\partial\sigma}f_2$ for every $j$-cell $\sigma\in K_{\delta}$, we can sum over all $j$-cells to obtain an estimate of the form $\int_{|K^j_{\delta}|}f_1\leq C\int_{|K_{\delta}^{j-1}|}f_2$, where the constant $C$ is independent of $\delta$.
\end{remark}

\hspace{3mm} Let $K_{\delta}$ be an $n$-dimensional cubical complex as in the conclusion of Lemma \ref{slicelemma}. For every $j\leq n$ and $s \in [0,\delta]$,  we define the map $\phi_{j,s}:|K^j|\to |K^j|$ by identifying each $j$-cell $\sigma\in K^j$ with $[-\delta,\delta]^j$, and setting
$$\phi_{j,s}(x):=\frac{\delta\cdot  x}{\max\{s,|x|_{\infty}\}}.$$
(Here we use the notation $|x|_{\infty}:=\max_{1\leq i\leq j}|x_i|$.) The family $\phi_{j,s}$ then interpolates between the identity at $s=\delta$ and a retraction to the $(j-1)$-skeleton at $s=0$. For $1<p<j$, the pullbacks $\phi_{j,s}^*f=f\circ \phi_{j,s}$ then define endomorphisms of $\mathcal{W}^{1,p}(K^j_{\delta},\mathbb{R})$, whose properties we record in the following lemma. (Compare \cite{HangLin}, Sections 4 and 5.)

\begin{lem}\label{philem} For $1<p<j\leq n$ and $f\in L^{\infty}\cap\mathcal{W}^{1,p}(K_{\delta}^j,\mathbb{R})$, the family $[0,\delta]\ni s\mapsto \phi_{j,s}^*f$ of pullbacks is a continuous path in $\mathcal{W}^{1,p}(K_{\delta}^j,\mathbb{R}))$, for which
\begin{equation}\label{phienest}
\int_{|K^j_{\delta}|}|d(\phi^*_{j,s}f)|^p\leq C(M)\left(\frac{\delta}{j-p}\int_{|K_{\delta}^{j-1}|}|df|^p+\int_{|K^j_{\delta}|}|df|^p\right)
\end{equation}
and
\begin{equation}\label{philpest}
\int_{|K^j_{\delta}|}|f-\phi_{j,s}^*f|^p\leq C(M)\delta^p\left(\frac{\delta}{j-p}\int_{|K^{j-1}_{\delta}|}|df|^p+\int_{|K^j_{\delta}|}|df|^p\right).
\end{equation}
\end{lem}

\begin{proof} Since the restriction of $\phi_{j,s}$ to the $(j-1)$-skeleton $|K_{\delta}^{j-1}|$ is given by the identity map for all $s$, it is enough to show that $\phi_{j,s}^*f$ defines a continuous path in $W^{1,p}(\sigma, \mathbb{R})$ on each $j$-cell $\sigma \cong [-\delta,\delta]^j$, and to establish (\ref{phienest}) and (\ref{philpest}), it is enough to show (per Remark \ref{combrk}) that the estimates
\begin{equation}\label{phienest2}
\int_{\sigma}|d(\phi_{j,s}^*f)|^p\leq C(M)\left(\frac{\delta}{j-p}\int_{\partial\sigma}|df|^p+\int_{\sigma}|df|^p\right)
\end{equation}
and
\begin{equation}\label{philpest2}
\int_{\sigma}|f-\phi_{j,s}^*f|^p\leq C(M)\delta^p\left(\frac{\delta}{j-p}\int_{\partial\sigma}|df|^p+\int_{\sigma}|df|^p\right)
\end{equation}
hold on every $j$-cell $\sigma$.

\hspace{3mm} To see this, note that, by the scale-invariance of (\ref{phienest}) and (\ref{philpest2}), it is enough to establish continuity and the desired estimates in the case $\delta=1$. Moreover, by virtue of the usual bi-Lipschitz correspondence 
$$[-1,1]^j\to B_1^j, \text{\hspace{3mm}} x\mapsto \frac{|x|_{\infty}}{|x|}x$$
between the unit cube and the unit ball, we can identify the maps $\phi_{j,s}^*f$ with the maps $f_s\in W^{1,p}(B_1^j,\mathbb{R})$ given by
$$f_s(x):=f(x/\max\{s,|x|\}),\text{ }s\in [0,1].$$
It is then straightforward to check that
\begin{equation}\label{enercomp}
\int_{B_1}|df_s|^p=\int_s^1 r^{j-1-p}\left(\int_{\partial B_1}|df|_{\partial B_1}|^p\right)dr+s^{j-p}\int_{B_1}|df|^p,
\end{equation}
from which (\ref{phienest2}) follows immediately. And since $f_s=f$ on $\partial B_1$, the $L^p$ estimate (\ref{philpest2}) follows from (\ref{phienest2}) and the $L^p$ Poincar\'{e} inequality (see, e.g., \cite{EG}, Section 4.5.2).

\hspace{3mm} To see the continuity of $s\mapsto f_s$ in $W^{1,p}(B_1)$, we appeal first to the fact that $f\in L^{\infty}$ and the bounded convergence theorem to conclude that $s\mapsto f_s$ is a continuous path in $L^p$. And since we see from (\ref{enercomp}) that the $p$-energy $s\mapsto E_p(f_s)$ is a continuous function of $s$, $f_s$ must give a continuous path in $W^{1,p}(B_1)$ as well.

\end{proof}

\hspace{3mm} Next, we define the retraction maps $\Phi_j: |K_{\delta}|\to |K_{\delta}^{j-1}|$, sending almost every point in the cubical complex into the $(j-1)$-skeleton, by 
\begin{equation}
\Phi_j:=\phi_{j,0}\circ \cdots \circ \phi_{n,0}.
\end{equation}
By definition of the radial retractions $\phi_{j,0}$, we observe that $\Phi_j$ is locally Lipschitz away from an $(n-j)$-dimensional set, called the \emph{dual $(n-j)$-skeleton} $L^{n-j}$ to $K$. 

\hspace{3mm} Note that the dual skeleton $L^{n-j}$ can be expressed as the union (disjoint except for a set of dimension $(n-j-1)$) over all $j$-cells $\sigma\in K$ of the sets
\begin{equation}\label{dualpdef}
P(\sigma):=\overline{\Phi_{j+1}^{-1}(a_{\sigma})},
\end{equation}
where $a_{\sigma}$ is the center of $\sigma$; and each $P(\sigma)$, in turn, can be decomposed into the union of the intersections
$$P(\sigma)\cap \Delta$$
over all $n$-cells $\Delta$ containing $\sigma$ as a $j$-face. Identifying a given $n$-cell $\Delta\in K_{\delta}$ with $[-\delta,\delta]^n$ and the $j$-cell $\sigma$ with $[-\delta,\delta]^j\times \{(\delta,\ldots,\delta)\}$, we see that $P(\sigma)\cap \Delta$ is given by the box $\{0\}\times [0,\delta]^{n-j}$. Since, per Remark \ref{combrk}, the number of $n$-cells intersecting a given $j$-cell is bounded independent of $\delta$ for the cubeulations $K_{\delta}$ of Lemma \ref{slicelemma}, it follows in this case that
\begin{equation}\label{dualcellsest}
\mathcal{H}^{n-j}(P(\sigma))\leq C(M)\delta^{n-j}
\end{equation}
for every $j$-cell $\sigma\in K_{\delta}$.

\begin{remark}\label{homrk} Observe also that if $\Delta\in K_{\delta}$ is an $n$-face of $K_{\delta}$, $\sigma$ is a $j$-cell in $\partial \Delta$, and $f: S^{j-1}\to \Delta\setminus L^{n-j}$ is an embedding of the $(j-1)$-sphere $S^{j-1}$ that links with $P(\sigma)$ (e.g., sending $S^{j-1}$ to the sphere $S_{\delta/2}^{j-1}\times \{(\delta/2,\ldots,\delta/2)\}$, under the identifications $\Delta\cong[-\delta,\delta]^n$, $\sigma\cong[-\delta,\delta]^j\times \{(\delta,\ldots,\delta)\}$) then $f$ can be homotoped through $\Delta\setminus L^{n-j}$ to a homeomorphism with $\partial \sigma$.
\end{remark}

\hspace{3mm} As in \cite{HangLin}, we will make repeated use of the retraction maps $\Phi_j$ to deform given maps in $W^{1,p}(M,N)$ to ones which are continuous away from $L^{n-j}$. For the arguments of Section \ref{limthms} in particular, it will be important to have some precise estimates for the maps produced in this way, which we record in the following lemma.

\begin{lem}\label{wkapplem}
For $f\in W^{1,p}(M,\mathbb{R})$ and $\delta\in (0,1]$, choose a cubeulation $h:|K_{\delta}|\to M$ as in the conclusion of Lemma \ref{slicelemma}. Then for $k\in \{2,\ldots,n\}$ and $1<p<k$, the function $\tilde{f}:=f|_{|K^{k-1}_{\delta}|}\circ \Phi_k\circ h^{-1}$ satisfies
\begin{equation}\label{wkappenest}
E_p(\tilde{f})\leq \frac{C(M)}{k-p}E_p(f)
\end{equation}
and
\begin{equation}\label{wkapplpest}
\int_M|f-\tilde{f}|^p\leq \delta^p\frac{C(M)}{k-p}E_p(f).
\end{equation}
\end{lem}

\begin{proof} Since $K_{\delta}$ is chosen according to Lemma \ref{slicelemma}, we know that $f\in \mathcal{W}^{1,p}(K_{\delta},\mathbb{R})$, and
$$\int_{|K_{\delta}^j|}|df|^p\leq C(M)\delta^{j-n}E_p(f)$$
for every $0\leq j\leq n$. Since $\tilde{f}=f\circ \phi_{k,0}$ on the $k$-skeleton $|K^k|$, we can then apply the estimates of Lemma \ref{philem} to conclude that
\begin{eqnarray*}
\int_{|K^k_{\delta}|}|d\tilde{f}|^p&\leq &C(M)\left(\frac{\delta}{k-p}\cdot \delta^{k-1-n}E_p(f)+\delta^{k-n}E_p(f)\right)\\
&\leq &C'(M)\frac{\delta^{k-n}}{k-p}E_p(f)
\end{eqnarray*}
and
\begin{equation}\label{kskellp}
\int_{|K_{\delta}^k|}|f-\tilde{f}|^p\leq C'(M)\delta^p\cdot\frac{\delta^{k-n}}{k-p}E_p(f).
\end{equation}

\hspace{3mm} Next, we can apply the preceding estimates together with the conclusion of Lemma \ref{philem} on the $(k+1)$-skeleton, to see that $\tilde{f}|_{|K^{k+1}|}=\phi_{k+1,0}^*\phi_{k,0}^*f$ satisfies
\begin{eqnarray*}
\int_{|K^{k+1}_{\delta}|}|d\tilde{f}|^p&\leq &C(M)\left(\frac{\delta}{k+1-p}\cdot \frac{\delta^{k-n}}{k-p}E_p(f)+\delta^{k+1-n}E_p(f)\right)\\
&\leq &C'(M)\frac{\delta^{k+1-n}}{k-p}E_p(f).
\end{eqnarray*}
To estimate $|f-\tilde{f}|^p$ on $|K^{k+1}|$, we again apply the scaled $L^p$ Poincar\'{e} inequality
$$\int_{\sigma}|f-\tilde{f}|^p\leq C\left(\delta^p\int_{\sigma}|d(f-\tilde{f})|^p+\delta \int_{\partial \sigma}|f-\tilde{f}|^p\right)$$
to every $(k+1)$-cell $\sigma \cong [-\delta,\delta]^{k+1}$ in $K_{\delta}$, and sum over $\sigma$ (again appealing to Remark \ref{combrk}), to obtain
\begin{eqnarray*}
\int_{|K^{k+1}_{\delta}|}|f-\tilde{f}|^p&\leq &C(M)\left(\delta^p\int_{|K^{k+1}|}(|df|^p+|d\tilde{f}|^p)+\delta\int_{|K^k_{\delta}|}|f-\tilde{f}|^p\right)\\
&\leq &C'(M)\left(\delta^p[\delta^{k+1-n}E_p(f)+\frac{\delta^{k+1-n}}{k-p}E_p(f)]+\delta\cdot \delta^p \frac{\delta^{k-n}}{k-p}E_p(f)\right).
\end{eqnarray*}
In particular, we conclude that
$$\int_{|K_{\delta}^{k+1}|}|d\tilde{f}|^p\leq C(M)\frac{\delta^{k+1-n}}{k-p}E_p(f)$$
and
$$\int_{|K_{\delta}^{k+1}|}|f-\tilde{f}|^p\leq C(M)\delta^p\frac{\delta^{k+1-n}}{k-p}E_p(f).$$
Carrying on by induction on $j$, for each $j$-skeleton $|K_{\delta}^j|$ with $j\geq k$, we find that
$$\int_{|K_{\delta}^j|}|d\tilde{f}|^p\leq \frac{C}{k-p}\delta^{j-n}E_p(f)$$
and
$$\int_{|K_{\delta}^j|}|f-\tilde{f}|^p\leq \frac{C}{k-p}\delta^p\delta^{j-n}E_p(f)$$
for every $k\leq j\leq n$; in particular, taking $j=n$, we obtain the desired estimates for $\tilde{f}$.

\end{proof}

\subsection{Homological Singularities}\label{homsingbasics}\hspace{20mm}

\hspace{3mm} In this section, we define the homological singularities of Sobolev maps associated to real cohomology classes in the target manifold $N$. All of the results of this section are contained, with slightly different terminology, in Section 5.4.2 of \cite{GMS2}, but we opt for a self-contained treatment more directly suited to the purposes of this paper. 

\hspace{3mm} First, we fix some notation from the theory of currents (see \cite{Fed}, \cite{GMS1}, or \cite{Sim} for an introduction). Denote by $\Omega^m(M)$ the space of smooth $m$-forms on a compact manifold $M$, and by $\Omega_c^m(U)$ the space of compactly supported $m$-forms in an open set $U$. Following  \cite{GMS1}, we use $L_m^q(M)$ to denote the closure of $\Omega^m(M)$ with respect to the $L^q$ norm. The space of general $m$-currents will be denoted by $\mathcal{D}_m(M)$, and for each $T\in \mathcal{D}_m(M)$, following \cite{Sim}, we define the mass $\mathbb{M}(T)$ by
$$\mathbb{M}(T):=\sup\{\langle T,\zeta\rangle \mid  \zeta\in\Omega^m(M),\text{ }\|\zeta\|_{L^{\infty}}\leq 1\}.$$
We use $\mathcal{I}_m(M;\mathbb{Z})$ to denote the space of integer rectifiable $m$-currents, and $\mathcal{Z}_m(M;\mathbb{Z})$ for the subspace of integral $m$-cycles.

\hspace{3mm} Now, let $N$ be a closed, oriented Riemannian manifold. For every integer $1\leq m\leq \dim N$, denote by $\mathcal{A}^m(N)$ the collection of closed $m$-forms on $N$ satisfying
\begin{equation}\label{intcond}
\langle \Sigma,\alpha\rangle\in \mathbb{Z}\text{ for every }\Sigma\in \mathcal{Z}_m(N;\mathbb{Z}).
\end{equation}
Observe that the image of $\mathcal{A}^m(N)$ in de Rham cohomology defines a lattice of full rank in $H^m_{dR}(N)$. Indeed, given integral $m$-cycles $\Sigma_1,\ldots,\Sigma_q$ in $N$ generating $H_m(N;\mathbb{R})$, we can find corresponding cohomology classes $[\alpha_1],\ldots,[\alpha_m]\in H_{dR}^m(N)$ for which $\langle \Sigma_i,\alpha_j\rangle=\delta_{ij}$ (see, e.g., \cite{GMS1}, Section 5.4.1). These $\alpha_i$ evidently lie in $\mathcal{A}^m(N)$, and give a basis for $H^m_{dR}(N)$. 

\hspace{3mm} We now fix some $k\in \{2,\ldots,\dim N+1\}$, and $\alpha\in \mathcal{A}^{k-1}(N)$. Appealing to Nash's embedding theorem, we also fix an isometric embedding $N\subset \mathbb{R}^L$ of $N$ into some Euclidean space. We can then easily extend our $(k-1)$-form $\alpha$ to a compactly supported form 
$$\bar{\alpha}=\Sigma_{|I|=k-1}\bar{\alpha}_I(x)dx^I\in \Omega_c^{k-1}(\mathbb{R}^L),$$
for instance by taking the pullback $\pi_N^*\alpha$ of $\alpha$ to a tubular neighborhood of $N$ by the nearest-point projection $\pi_N$, then multiplying by a suitable cut-off function. 

\hspace{3mm} Now, let $M^n$ be a compact, oriented manifold, possibly with boundary. We record next some important estimates for the pullback $u^*(\overline{\alpha})$ of $\bar{\alpha}$ by smooth maps $u\in C^{\infty}(M,\mathbb{R}^L)$.

\begin{lem}\label{pbestlem} For $u,v\in C^{\infty}(M,\mathbb{R}^L)$, there exist a $(k-1)$-form $\beta(u,v)\in \Omega^{k-1}(M)$ and a $(k-2)$-form $\eta(u,v)\in \Omega^{k-2}(M)$ such that
\begin{equation}\label{pbdecomp}
v^*(\bar{\alpha})-u^*(\bar{\alpha})=\beta+d\eta,
\end{equation}
and the following pointwise estimates hold:
\begin{multline}\label{pbest1}
|v^*(\bar{\alpha})-u^*(\bar{\alpha})| \leq C(\alpha)|u-v|(|du|^{k-1}+|dv|^{k-1})\\
+C(\alpha)|du-dv|(|du|^{k-2}+|dv|^{k-2}),
\end{multline}
\begin{equation}\label{pbest2}
|\beta(u,v)|\leq C(\alpha)|u-v|(|du|^{k-1}+|dv|^{k-1}),
\end{equation}
and
\begin{equation}\label{pbest3}
|\eta(u,v)|\leq C(\alpha)|u-v|(|du|^{k-2}+|dv|^{k-2}).
\end{equation}
\end{lem}

\begin{remark} Here, $C(\alpha)$ denotes a constant depending on $\|\overline{\alpha}\|_{C^1}$.
\end{remark}

\begin{proof} Write
$$v^*(\bar{\alpha})-u^*(\bar{\alpha}):=\Sigma_I (\bar{\alpha}_I(v)dv^I-\bar{\alpha}_I(u)du^I).$$
Fixing a multi-index $I=(i_1,\ldots,i_{k-1})$, we begin by rearranging
$$\alpha_I(v)dv^I-\alpha_I(u)du^I=(\alpha_I(v)-\alpha_I(u))dv^I+\alpha_I(u)(dv^I-du^I),$$
and noting that
$$|\alpha_I(v)-\alpha_I(u)||dv^I|\leq \|\nabla \alpha\|_{L^{\infty}}|u-v||dv|^{k-1}.$$
The first estimate (\ref{pbest1}) follows immediately, and absorbing the terms
$$(\alpha_I(v)-\alpha_I(u))dv^I$$
into $\beta(u,v)$, we see that, to complete the proof of (\ref{pbdecomp}), it suffices to exhibit a decomposition of the form (\ref{pbdecomp}) for the remaining terms $\alpha_I(u)(dv^I-du^I)$.

\hspace{3mm} To this end, writing $I_2$ for the multi-index $(i_2,\ldots,i_{k-1})$, we observe that
\begin{eqnarray*}
\alpha_I(u)(dv^I-du^I)&=&\alpha_I(u)(d(v^{i_1}-u^{i_1})\wedge dv^{I_2}+du^{i_1}\wedge (dv^{I_2}-du^{I_2}))\\
&=&d[\alpha_I(u)(v^{i_1}-u^{i_1})\wedge dv^{I_2}]-(v^{i_1}-u^{i_1})d(\alpha_I(u))\wedge dv^{I_2}\\
&&+du^{i_1}\wedge \alpha_I(u)(dv^{I_2}-du^{I_2}).
\end{eqnarray*}
Now, the $(k-2)$-form $\alpha_I(u)(v^{i_1}-u^{i_1})\wedge dv^{I_2}$ evidently satisfies an estimate of the form (\ref{pbest3}), and so can be absorbed into $\eta(u,v)$, while the $(k-1)$-form $(v^{i_1}-u^{i_1})d(\alpha_I(u))\wedge dv^{I_2}$ can likewise be absorbed into $\beta(u,v)$. To deal with the leftover term 
$$du^{i_1}\wedge \alpha_I(u)(dv^{I_2}-du^{I_2}),$$
we apply the same argument to the $(k-2)$-form $\alpha_I(u)(dv^{I_2}-du^{I_2})$ that we did to the $(k-1)$-form $\alpha_I(u)(dv^I-du^I)$, and carrying on in this way, we eventually arrive at the desired decomposition.
\end{proof}

\hspace{3mm} Integrating the estimates (\ref{pbest1})-(\ref{pbest3}) of Lemma \ref{pbestlem} and making liberal use of H\"{o}lder's inequality, we obtain for any $p>k-1$ the bounds
\begin{eqnarray*}
\|v^*(\bar{\alpha})-u^*(\bar{\alpha})\|_{L^1}&\leq & C(\alpha)(\|du\|_{L^p}^{k-1}+\|dv\|_{L^p}^{k-1})\|u-v\|_{L^{\infty}}^{k-p}\|u-v\|_{L^p}^{p+1-k}\\
&&+(\|du\|_{L^{k-1}}^{k-2}+\|dv\|_{L^{k-1}}^{k-2})\|du-dv\|_{L^{k-1}},
\end{eqnarray*}
\begin{equation}\label{pbintest2}
\|\beta(u,v)\|_{L^1}\leq C(\alpha)(\|du\|_{L^p}^{k-1}+\|dv\|_{L^p}^{k-1})\|u-v\|_{L^{\infty}}^{k-p}\|u-v\|_{L^p}^{p+1-k},
\end{equation}
and
\begin{equation}\label{pbintest3}
\|\eta(u,v)\|_{L^1}\leq C(\alpha)(\|du\|_{L^{k-1}}^{k-2}+\|dv\|_{L^{k-1}}^{k-2})\|u-v\|_{L^{k-1}}.
\end{equation}

\hspace{3mm} For the remainder of this section, we will have $p\in (k-1,k)$. It then follows from the estimates above that the pullback assignment
$$u\mapsto u^*(\alpha)$$
gives a well-defined, continuous map from $W^{1,p}(M^n,N)$ to the space $L^1_{k-1}(M)$ of $(k-1)$-forms with coefficients in $L^1$. For any map $u\in W^{1,p}(M,N)$, we can, in particular, define the $(n+1-k)$-current $S_{\alpha}(u)\in \mathcal{D}_{n+1-k}(M)$ dual to $u^*(\alpha)$ by
\begin{equation}
\langle S_{\alpha}(u),\zeta\rangle:=\int_Mu^*(\alpha)\wedge \zeta.
\end{equation} 
By virtue of (\ref{pbintest2}) and (\ref{pbintest3}), we then have the following decomposition lemma for the difference $S_{\alpha}(v)-S_{\alpha}(u)$.

\begin{lem}\label{sdecomplem} For $u,v\in W^{1,p}(M,N)$, the difference $S_{\alpha}(v)-S_{\alpha}(u)$ admits a decomposition of the form
\begin{equation}
S_{\alpha}(v)-S_{\alpha}(u):=S_{\alpha}(u,v)+\partial R_{\alpha}(u,v),
\end{equation}
for some $R_{\alpha}(u,v)\in \mathcal{D}_{n-k+2}(M)$ and $S_{\alpha}(u,v)\in \mathcal{D}_{n+1-k}(M)$ satisfying the mass bounds
\begin{equation}
\mathbb{M}(S_{\alpha}(u,v))\leq C(\alpha)[E_p(u)^{\frac{k-1}{p}}+E_p(v)^{\frac{k-1}{p}}]\|u-v\|_{L^p}^{1+p-k}.
\end{equation}
and
\begin{equation}
\mathbb{M}(R_{\alpha}(u,v))\leq C(\alpha)[\|du\|_{L^{k-1}}^{k-2}+\|dv\|_{L^{k-1}}^{k-2}]\|u-v\|_{L^{k-1}}.
\end{equation}
\end{lem}

\hspace{3mm} For $u\in W^{1,p}(M,N)$, we now define the \emph{homological singularity $T_{\alpha}(u)\in \mathcal{D}_{n-k}(M)$ associated to $\alpha$} to be the $(n-k)$-boundary
\begin{equation}
\langle T_{\alpha}(u),\zeta\rangle:=\langle \partial S_{\alpha}(u),\zeta\rangle=\int_M u^*(\alpha)\wedge d\zeta.
\end{equation}
Homological singularities of this sort have been been considered by various authors--see, for instance, \cite{GMS2}, or \cite{BCDH}, which studies their role as obstructions to the approximation of Sobolev maps by smooth maps. In the special case $N=S^{k-1}$, $\alpha=\frac{dvol_{S^{k-1}}}{\sigma_{k-1}}$, the current $T_{\alpha}$ coincides with the distributional Jacobian, whose geometric properties have been well studied in recent decades (see, for instance, \cite{ABOpres}, \cite{JSbnv}, and references therein).

\hspace{3mm} When $u$ is smooth on the support of an $(n-k)$-form $\zeta\in \Omega_c^{n-k}(\mathring{M})$ supported in the interior $\mathring{M}$ of $M$, Stokes's theorem and the naturality of the exterior derivative give
\begin{eqnarray*}
\langle T_{\alpha}(u),\zeta\rangle&=&\int u^*(\alpha)\wedge d\zeta\\
&=&(-1)^{k-1}\int d(u^*(\alpha))\wedge\zeta\\
&=&0.
\end{eqnarray*}
In fact, if $u$ is continuous on an open set containing $spt(\zeta)$, then one again has
$$\langle T_{\alpha}(u),\zeta\rangle=0,$$
since we can find a sequence of smooth maps $u_j\in C^{\infty}(M,N)$ approaching $u$ in $W^{1,p}$ on a neighborhood of $spt(\zeta)$ (see, e.g., \cite{Bet},\cite{HangLin}). In particular, if $u$ is continuous away from a closed set $Sing(u)\subset M$, it follows that
\begin{equation}\label{sptchar}
spt(T_{\alpha}(u))\subset Sing(u)\cup \partial M.
\end{equation}

\hspace{3mm} Next, define 
$$\mathcal{E}^p(M,N)\subset W^{1,p}(M,N)$$
to be the collection of maps $u\in W^{1,p}(M,N)$ of the form
$$u=f\circ \Phi_k\circ h^{-1}$$
for some cubeulation $h:|K|\to M$ of $M$ and some Lipschitz map $f\in Lip(|K^{k-1}|,N)$ from the $(k-1)$-skeleton. For such maps, the set $Sing(u)$ of discontinuities is evidently contained in the dual $(n-k)$-skeleton $L^{n-k}$ to $K$, and the homological singularity $T_{\alpha}(u)$ is given by an integral $(n-k)$-cycle that we can describe explicitly. 

\begin{prop}\label{epint}\emph{(cf. [GMS2] Section 5.4.2, Theorem 1)} If $u\in \mathcal{E}^p(M,N)$ is given by $u=f\circ \Phi_k\circ h^{-1}$ for some $f\in Lip(|K^{k-1}|,N)$ and a cubeulation $h:|K|\to M$, then for any $(n-k)$-form $\zeta\in\Omega_c^{n-k}(\mathring{M})$ supported in the interior of $M$, the pairing with $T_{\alpha}(u)$ is given by
\begin{equation}\label{rptchar}
\langle T_{\alpha}(u),\zeta\rangle=\Sigma_{\sigma\in K^k\setminus K^{k-1}}\theta(\sigma)\cdot \int_{P(\sigma)}\zeta,
\end{equation}
where
\begin{equation}\label{thetadef1}
|\theta(\sigma)|=|\int_{\partial \sigma}f^*(\alpha)|,
\end{equation}
and $P(\sigma)$ is defined as in (\ref{dualpdef}) to be the component of $L^{n-k}$ intersecting the $k$-cell $\sigma\in K$.
\end{prop}
\begin{remark} The integrality $\theta(\sigma)\in \mathbb{Z}$ follows from the fact that $\alpha\in \mathcal{A}^{k-1}(N)$, since the pushforward $f_*(\partial\sigma)$ of the $(k-1)$-cycle $\partial\sigma$ by the Lipschitz map $f$ is an integral $(k-1)$-cycle in $N$.
\end{remark}

\hspace{3mm} The proposition follows from results in Section 5.4.2 of \cite{GMS2}, but in the interest of keeping the discussion self-contained (and because our terminology differs somewhat from that of \cite{GMS2}) we provide a proof below. In fact, the conclusion of Proposition \ref{epint} applies to a much larger collection of maps than $\mathcal{E}^p(M,N)$--namely, any $W^{1,p}$ map which is continuous away from the dual $(n-k)$-skeleton of some cubeulation (see \cite{GMS2}). 

\begin{proof} To begin, we claim that it is enough to establish (\ref{rptchar}) for forms $\zeta\in \Omega_c^{n-k}(\mathring{M})$ supported in the interior of a single $n$-face $\Delta$ of the cubeulation. To see this, denote by $\Xi$ the $(n-k-1)$-dimensional intersection
$$\Xi:=L^{n-k}\cap |K^{n-1}|,$$
and for $\epsilon\in (0,1/2)$, define the cutoff functions
$$\chi_{\epsilon}(x):=\psi(\epsilon^{-1}dist(x,L^{n-k}))$$
and
$$\varphi_{\epsilon}(x):=\psi(\epsilon^{-1}dist(x,\Xi)),$$
where $\psi\in C^{\infty}(\mathbb{R})$ satisfies 
\begin{equation}\label{psichar}
\psi(t)=0\text{ for }t\geq 1\text{ and }\psi(t)=1\text{ for }t\leq \frac{1}{2}.
\end{equation}
Then $\chi_{\epsilon}$ is supported on the $\epsilon$-neighborhood of $L^{n-k}$ and $\chi_{\epsilon}\equiv 1$ near $L^{n-k}$, while $\varphi_{\epsilon}$ is supported on the $\epsilon$-neighborhood of $\Xi$, with $\varphi_{\epsilon}\equiv 1$ near $\Xi$.

\hspace{3mm} For any $\zeta\in \Omega_c^{n-k}(M)$, it follows from (\ref{sptchar}) that
$$\langle T_{\alpha}(u),\zeta\rangle=\langle T_{\alpha}(u),\chi_{\epsilon}\zeta\rangle.$$
For $\epsilon>0$ sufficiently small, we observe that the form
$$\chi_{\epsilon}\zeta-\varphi_{\epsilon}\zeta$$
is supported away from the $(n-1)$-skeleton $|K^{n-1}|$, and can therefore be written as a sum of forms supported in the interiors of the $n$-faces $\Delta$ of $K$. In particular, to justify our claim that it suffices to establish (\ref{rptchar}) for forms $\zeta$ supported in an $n$-face $\Delta$, it is enough to show that
\begin{equation}\label{smallskelvan}
\lim_{\epsilon\to 0}\langle T_{\alpha}(u),\varphi_{\epsilon}\zeta\rangle=0.
\end{equation}

\hspace{3mm} To establish (\ref{smallskelvan}), we first observe that
\begin{eqnarray*}
|\langle T_{\alpha}(u),\varphi_{\epsilon}\zeta\rangle|&=&|\int u^*(\alpha)\wedge d\varphi_{\epsilon}\wedge \zeta+\int u^*(\alpha)\wedge \varphi_{\epsilon}d\zeta|\\
&\leq & C\left(\frac{1}{\epsilon}\|\zeta\|_{L^{\infty}}+\|d\zeta\|_{L^{\infty}}\right)\int_{\{\dist_{\Xi}\leq \epsilon\}}|du|^{k-1}.
\end{eqnarray*}
Since $u=f\circ \Phi_k$ for $f\in Lip(|K^{k-1}|,N)$, we have almost everywhere a gradient estimate of the form
$$|du|(x)\leq C\frac{Lip(f)}{\dist_{L^{n-k}}(x)},$$
and we can check by direct computation on each $n$-cell $\Delta$ that
\begin{eqnarray*}
\int_{\{\dist_{\Xi}\leq \epsilon\}}|du|^{k-1}&\leq &CLip(f)^{k-1}\int_{\{\dist_{\Xi}\leq \epsilon\}}(\dist_{L^{n-k}}(x))^{1-k}d\mathcal{H}^n(x)\\
&\leq &C Lip(f)^{k-1}\cdot \epsilon^2.
\end{eqnarray*}
Returning to our estimate for $\langle T_{\alpha}(u),\varphi_{\epsilon}\zeta\rangle$, we then see that
\begin{eqnarray*}
\lim_{\epsilon\to 0}|\langle T_{\alpha}(u),\varphi_{\epsilon}\zeta\rangle|&\leq &\lim_{\epsilon \to 0}C\left(\frac{1}{\epsilon}\|\zeta\|_{L^{\infty}}+\|d\zeta\|_{L^{\infty}}\right)\cdot Lip(f)^{k-1}\epsilon^2\\
&\leq & C(K,\zeta,f)\lim_{\epsilon\to 0}\epsilon,
\end{eqnarray*}
so (\ref{smallskelvan}) holds, and we can restrict our attention to forms $\zeta$ supported in the interior of a single $n$-cell $\Delta$.

\hspace{3mm} In fact, if we modify the definition of $\Xi$ above by adding the $(n-k-1)$-dimensional set given by the union of all intersections $P(\sigma_1)\cap P(\sigma_2)$ for distinct $k$-cells $\sigma_1, \sigma_2\in K$, then the same argument shows that it is enough to establish (\ref{rptchar}) for $\zeta$ supported in the interior of $\Delta\cap \Phi_{k+1}^{-1}(\sigma)$ for some $n$-face $\Delta\in K$ and $k$-cell $\sigma\in K$. 

\hspace{3mm} Thus, identifying $\Delta$ homothetically with $I^n:=[-1,1]^n$, $\sigma$ with $\{(1,\ldots,1)\}\times I^k$, and (consequently) $\Delta\cap \Phi_{k+1}^{-1}(\sigma)$ with $[0,1]^{n-k}\times I^k$, it remains to show that for a $W^{1,p}$ map
$$u:E=(0,1)^{n-k}\times I^k\to N$$
with
$$Sing(u)\subset [0,1]^{n-k}\times \{0\},$$
and any $\zeta\in \Omega_c^{n-k}((0,1)^{n-k}\times (-1,1)^k),$ we have
\begin{equation}\label{simprptchar}
\langle T_{\alpha}(u),\zeta\rangle=\theta\int_{[0,1]^{n-k}\times \{0\}}\zeta,\text{\hspace{2mm} where\hspace{2mm} }|\theta|=|\int_{\partial \sigma}u^*(\alpha)|.
\end{equation}

\hspace{3mm} Since $\partial T_{\alpha}(u)=0$ and the support $spt(T_{\alpha}(u))$ satisfies (by (\ref{sptchar}))
$$spt(T_{\alpha}(u))\cap E\subset [0,1]^{n-k}\times \{0\},$$
it follows from standard constancy theorems (e.g., Theorem 2 in Section 5.3.1 of \cite{GMS1}) that $T_{\alpha}(u)$ has the form (\ref{simprptchar}) for \emph{some} $\theta\in \mathbb{R}$, provided that
\begin{equation}\label{perpcond}
\langle T_{\alpha}(u),\zeta\rangle=0\text{ for every }\zeta\in \Omega_c^{n-k}(E)\text{ with }\langle \zeta, dy^1\wedge \cdots \wedge dy^{n-k}\rangle\equiv 0.
\end{equation}
To prove the orthogonality condition (\ref{perpcond}), write $(y^1,\ldots,y^{n-k},z^1,\ldots,z^k)$ for the coordinates of $E$, and consider $\zeta\in \Omega_c^{n-k}(E)$ of the form
\begin{equation}\label{perptestform}
\zeta=dz^j\wedge \omega\text{ for }\omega\in \Omega_c^{n-k-1}(E),
\end{equation}
and let $\chi\in C_c^{\infty}((-\delta,\delta))$ be a bump function with $\chi(t)=1$ for $t\in [-\frac{\delta}{2},\frac{\delta}{2}]$. Since $spt(T_{\alpha}(u))\subset\{(y,z)\mid z=0\},$ we then have
\begin{eqnarray*}
|\langle T_{\alpha}(u),\zeta\rangle|&=&|\langle T_{\alpha}(u),\chi(z^j)dz^j\wedge \omega\rangle|\\
&=&|\int u^*(\alpha)\wedge \chi(z^j)dz^j\wedge d\omega|\\
&\leq &C\|d\omega\|_{L^{\infty}}E_p(u)^{\frac{k-1}{p}}Vol(\{|z|<\delta\})^{1-\frac{k-1}{p}}.
\end{eqnarray*}
Since $\delta>0$ was arbitrary, we can then take $\delta\to 0$, to see that $\langle T_{\alpha}(u),\zeta\rangle=0$ for any $\zeta$ of the form (\ref{perptestform}). In particular, it follows that (\ref{perpcond}) holds, so that $T_{\alpha}(u)$ indeed has the form (\ref{simprptchar}) for some $\theta\in \mathbb{R}$.

\hspace{3mm} To determine the constant $\theta$ in (\ref{simprptchar}), we test $T_{\alpha}(u)$ against a form
$$\zeta(x)=\zeta(y,z)=\varphi(y)\psi(|z|)dy^1\wedge \cdots\wedge dy^{n-k},$$
where $\varphi\in C_c^{\infty}((0,1)^{n-k})$ and $\psi$ is given by (\ref{psichar}). By direct computation, we see that
\begin{eqnarray*}
\langle T_{\alpha}(u),\zeta\rangle&=&\int u^*(\alpha)\wedge d\zeta\\
&=&\int \varphi(y)u^*(\alpha)\wedge\psi'(|z|)d|z|\wedge dy^1\wedge \cdots \wedge dy^{n-k}\\
&=&(-1)^{k(n-k)}\int_{y\in (0,1)^{n-k}}\varphi(y)\left(\int_{y\times (-1,1)^k}\psi'(|z|)u^*(\alpha)\wedge d|z|\right)dy\\
&=&(-1)^{k(n-k)}(-1)^{k-1}\int_{y\in (0,1)^{n-k}}\varphi(y)\int_0^1\psi'(r)\left(\int_{y\times S_r^{k-1}(0)}u^*(\alpha)\right)dr dy.
\end{eqnarray*}
Now, since $u$ is locally Lipschitz away from $[0,1]^{n-k}\times \{0\}$, it follows from the observations in Remark \ref{homrk} that 
$$\int_{y\times S_r^{k-1}(0)}u^*(\alpha)=\int_{\partial\sigma}u^*(\alpha)$$
for every sphere $y\times S_r^{k-1}(0)$ linking with $[0,1]^{n-k}\times \{0\}$. Using this in the preceding computation, we see that
\begin{eqnarray*}
\langle T_{\alpha}(u),\zeta\rangle&=&(-1)^{k(n-k+1)-1}\langle \alpha, u_*(\partial\sigma)\rangle\int_{y\in (0,1)^{n-k}}\varphi(y)dy \int_0^1\psi'(r)dr\\
&=&(-1)^{k(n-k+1)}\langle \alpha,u_*(\partial\sigma)\rangle\int_{[0,1]^{n-k}\times\{0\}}\zeta.
\end{eqnarray*}
Thus, the constant $\theta$ in (\ref{simprptchar}) must be given by 
$$\theta=(-1)^{k(n-k+1)}\int_{\partial\sigma}u^*(\alpha),$$
as desired.
\end{proof}

\section{Limits of Homological Singularities as $p\to k$}\label{limthms}

\subsection{Degree-type Estimates in $k$-Dimensional Domains}\label{degests}\hspace{30mm}

\hspace{3mm} In this section, we are concerned with estimating the topological quantity
$$\int_{\partial U}u^*(\alpha)$$
for maps $u\in W^{1,p}(U,N)\cap W^{1,p}(\partial U,N)$ on a $k$-dimensional domain $U\subset \mathbb{R}^k$ in terms of the $p$-energy $\int_U|du|^p$. Our arguments are modeled very closely on those used by Jerrard \cite{Jer} to estimate the degrees of $\mathbb{R}^k$-valued maps in terms of Ginzburg-Landau energies (see also \cite{SS},\cite{JSgl}). In the case $N=S^{k-1}$, $\alpha=\frac{dvol}{\sigma_{k-1}}$, estimates similar to the ones we consider here can also be found in \cite{HLW} (see also \cite{HardtLin}), where they are used to study the asymptotic behavior of $p$-energy minimizing maps from $U$ to $S^{k-1}$ as $p\to k$.

\hspace{3mm} Fix a closed $(k-1)$-form $\alpha\in \mathcal{A}^{k-1}(N)$ as before, and define the constant
\begin{equation}
\lambda(\alpha):=\sigma_{k-1}\sup\{\frac{\int_{S^{k-1}}u^*(\alpha)}{\int_{S^{k-1}}|du|^{k-1}}\mid u\in C^{\infty}(S^{k-1},N)\}.
\end{equation}
That $\lambda(\alpha)<\infty$ is clear from the estimates of Section \ref{homsingbasics}, and when working with specific examples, it is not difficult to obtain explicit bounds for $\lambda(\alpha)$. When $N=S^{k-1}$ and $\alpha$ is the normalized volume form $\frac{dvol_{S^{k-1}}}{\sigma_{k-1}}$, for example, one can check (see \cite{HLW}, Section 1) that 
$$\lambda(\alpha)=(k-1)^{\frac{1-k}{2}}.$$

\hspace{3mm} Next, for $p\in (k-1,k)$, we define the constants
$$c(N,\alpha,p):=\frac{\sigma_{k-1}}{\lambda(\alpha)^{\frac{p}{k-1}}},$$
and set
$$F_p(s):=\frac{c(N,\alpha,p)}{k-p}s^{k-p}.$$
The functions $F_p$ will take on the role in our setting played by the functions $\Lambda^{\epsilon}(s)$ in \cite{Jer}, \cite{JSgl}, \cite{SS}.
Since $0<k-p<1$, we easily check that
\begin{equation}\label{fprop1}
\frac{d}{ds}\left(\frac{F_p(s)}{s}\right)<0\text{ for }s>0,
\end{equation}
and, by the concavity of $s\mapsto s^{k-p},$ we have the subadditivity
\begin{equation}\label{fprop2}
F_p(s_1+s_2)\leq F_p(s_1)+F_p(s_2)
\end{equation}
for all $s_1,s_2>0$. Our estimates begin with the following simple lemma (compare \cite{Jer}, Proposition 3.2).

\begin{lem}\label{annlem}
Let $u\in C^{\infty}(B_{r_2}^k(0)\setminus B_{r_1}^k(0),N)$ be a smooth map from the annulus $B_{r_2}^k\setminus B_{r_1}^k$ to $N$, and set
$$d:=|\int_{\partial B_r^k(0)}u^*(\alpha)|$$
for some (hence every) $r\in [r_1,r_2]$. The $p$-energy of $u$ on $B_{r_2}\setminus B_{r_1}$ then satisfies the lower bound
\begin{equation}\label{annest}
E_p(u,B_{r_2}\setminus B_{r_1})\geq d [F_p(r_2/d)-F_p(r_1/d)].
\end{equation}
\end{lem}

\begin{proof} For any $r\in (r_1,r_2)$, by definition of $\lambda(\alpha)$, we know that
$$\sigma_{k-1}d\leq \lambda(\alpha)\int_{\partial B_r}|du|^{k-1}.$$
Raising both sides to the power $\frac{p}{k-1}$ and applying H\"{o}lder's inequality to the integral on the right-hand side, we then see that
\begin{eqnarray*}
(\sigma_{k-1}d)^{\frac{p}{k-1}}&\leq &\lambda(\alpha)^{\frac{p}{k-1}}|\partial B_r|^{\frac{p}{k-1}-1}\int_{\partial B_r}|du|^p\\
&=&\lambda(\alpha)^{\frac{p}{k-1}}\sigma_{k-1}^{\frac{p}{k-1}-1}r^{p+1-k}\int_{\partial B_r}|du|^p,
\end{eqnarray*}
which we can rearrange to read
$$\sigma_{k-1}d^{\frac{p}{k-1}}r^{k-p-1}\leq \lambda(\alpha)^{\frac{p}{k-1}}\int_{\partial B_r}|du|^p.$$

\hspace{3mm} Integrating the latter relation over $r\in [r_1,r_2]$, we arrive at the estimate
$$\sigma_{k-1}d^{\frac{p}{k-1}}\frac{r_2^{k-p}-r_1^{k-p}}{k-p}\leq \lambda(\alpha)^{\frac{p}{k-1}}E_p(u,B_{r_2}\setminus B_{r_1}).$$
The desired estimate (\ref{annest}) now follows from the trivial observation that
$$d^{\frac{p}{k-1}}\geq d\geq d^{1+p-k},$$
for $p\in (k-1,k)$, so that
$$d[F_p(r_2/d)-F_p(r_1/d)]=d^{1+p-k}\frac{r_2^{k-p}-r_1^{k-p}}{k-p}\leq d^{\frac{p}{k-1}}\frac{r_2^{k-p}-r_1^{k-p}}{k-p}.$$

\end{proof}

\hspace{3mm} We next record an analog of \cite{JSgl}, Proposition 6.4 (see also \cite{Jer}, Proposition 4.1 or \cite{SS}, Proposition 3.1), from which the main estimates of this section will follow.

\begin{lem}\label{ballgrowth} Let $U\subset \mathbb{R}^k$ be a bounded Lipschitz domain, and let $u:U\to N$ be smooth away from a finite set $\Sigma=\{a_1,\ldots,a_m\}\subset\subset U$ of singularities with
$$d_j:=\lim_{r\to 0}\int_{\partial B_r(a_j)}u^*(\alpha).$$
Then for every $\sigma>0$, there exists a family $\mathcal{B}(\sigma)=\{B_j^{\sigma}\}_{j=1}^{m(\sigma)}$ of $m(\sigma)\leq m$ disjoint closed balls of radius $r_j^{\sigma}$ such that, defining
$$d_j^{\sigma}:=|\Sigma_{a_{\ell}\in B_j^{\sigma}\cap\Sigma}d_{\ell}|,$$
we have
\begin{equation}\label{bg0}
\Sigma\subset \bigcup_{j=1}^{m(\sigma)}B_j^{\sigma}\text{\hspace{3mm} and\hspace{3mm} }\Sigma \cap B_j^{\sigma}\neq \varnothing\text{ for each }j,
\end{equation}
\begin{equation}\label{bg1}
\int_{U\cap B_j^{\sigma}}|du|^p\geq \frac{r_j^{\sigma}}{\sigma}F_p(\sigma)\text{ if }d_j^{\sigma}>0,
\end{equation}
and
\begin{equation}\label{bg2}
r_j^{\sigma}\geq \sigma d_j^{\sigma}\text{ if }B_j^{\sigma}\subset U.
\end{equation}
\end{lem}

\begin{proof} Denote by $S$ the collection of $\sigma>0$ for which such a family $\mathcal{B}(\sigma)$ exists. To see that $S$ is nonempty, for each $a_j\in \Sigma$, set
$$d_j:=\lim_{r\to 0}\int_{\partial B_r(a_j)}u^*(\alpha)$$
and
$$D:=1+\max_{1\leq j\leq m}|d_j|,$$
and choose $\sigma_0>0$ such that the balls $B_{D\sigma_0}(a_1),\ldots, B_{D\sigma_0}(a_m)$ are disjoint. Taking
$$r_j^{\sigma_0}:=\sigma_0 |d_j|\text{ if }d_j\neq 0,$$
and
$$r_j^{\sigma_0}:=\sigma_0\text{ if }d_j=0,$$
it's then clear that the collection
$$\mathcal{B}(\sigma_0):=\{B_{r_j^{\sigma_0}}(a_j)\}_{j=1}^m$$
satisfies (\ref{bg0}) and (\ref{bg2}), as well as (\ref{bg1}), by Lemma \ref{annlem}. In particular, $\sigma_0\in S$, so $S\neq \varnothing$.

\hspace{3mm} Since the functions $F_p(s)$ satisfy the growth conditions (\ref{fprop1}) and (\ref{fprop2}), we can apply Steps 2 and 3 in the proof of Proposition 6.4 of \cite{JSgl} directly (with $F_p$ in place of $\Lambda^{\epsilon}$) to see that the set $S$ is open, and closed away from $0$. In particular, we deduce that $S=(0,\infty)$, as desired. 

\end{proof}

\hspace{3mm} With Lemma \ref{ballgrowth} in hand, we arrive at the following proposition. (Compare \cite{Jer}, Theorem 1.2.)

\begin{prop}\label{jerrprop1} Let $U\subset \mathbb{R}^k$ and $u\in W^{1,p}(U,N)$ satisfy the hypotheses of Lemma \ref{ballgrowth}, and suppose that the singular set $\Sigma=\{a_1,\ldots,a_m\}$ satisfies
\begin{equation}\label{singbdrydist}
\inf_{1\leq j\leq m}dist(a_j,\partial U)\geq r>0.
\end{equation}
Setting
$$d:=|\int_{\partial U}u^*(\alpha)|=|\Sigma_{j=1}^md_j|,$$
we then have the lower bound
\begin{equation}\label{jpropest}
\int_U|du|^p\geq d\cdot F_p(r/2d)=\frac{c(N,\alpha,p)}{k-p}(r/2d)^{k-p}d.
\end{equation}
\end{prop}

\begin{proof} Again, we can argue just as in \cite{Jer}, \cite{JSgl}. Suppose that (\ref{jpropest}) doesn't hold, to obtain a contradiction. Setting $\sigma=\frac{r}{2d}$, we then have
$$\int_U|du|^p<d \cdot F_p(r/2d)=\frac{r}{2}\frac{F_p(\sigma)}{\sigma}.$$
Choosing a collection of balls $\mathcal{B}(\sigma)=\{B_j^{\sigma}\}$ according to Lemma \ref{ballgrowth}, it follows from (\ref{bg1}) that
\begin{equation}
r_j^{\sigma}\leq \frac{\sigma}{F_p(\sigma)}\int_{U\cap B_j^{\sigma}}|du|^p<\frac{r}{2}
\end{equation}
whenever $d_j^{\sigma}> 0$. In particular, if $d_j^{\sigma}>0$, we then deduce from (\ref{bg0}) and (\ref{singbdrydist}) that
$$B_j^{\sigma}\subset U,$$
and therefore, by (\ref{bg2}), we have
$$r_j^{\sigma}\geq \sigma d_j^{\sigma}.$$
Finally, summing (\ref{bg1}) over $1\leq j\leq m(\sigma)$, we see that
\begin{eqnarray*}
\int_U|du|^p&\geq &\Sigma_{j=1}^{m(\sigma)} \int_{U\cap B_j^{\sigma}}|du|^p\\
&\geq &\Sigma_{j=1}^{m(\sigma)} \frac{r_j^{\sigma}}{\sigma}F_p(\sigma)\\
&\geq &F_p(\sigma)\Sigma_{j=1}^{m(\sigma)} d_j^{\sigma}\\
&\geq &d\cdot F_p(\sigma),
\end{eqnarray*}
a contradiction. Thus, (\ref{jpropest}) holds.
\end{proof}

\begin{remark}\label{jprophyps} By the density results of Bethuel (namely, \cite{Bet}, Theorem 2), we can remove the requirement that $u$ have finite singular set from the hypotheses of Proposition \ref{jerrprop1}: the conclusion applies to any map $u\in W^{1,p}(U,N)$ for which $u$ is continuous on the $r$-neighborhood of $\partial U$ in $U$.
\end{remark}

\hspace{3mm} The final estimate of this section is a simple consequence of Proposition \ref{jerrprop1}, modeled on (\cite{ABOvar}, Lemma 3.10). Arguing much as in \cite{ABOvar}, we will employ this estimate repeatedly in the following sections to obtain the needed compactness results as $p\to k$ for the homological singularities $T_{\alpha}(u_p)$ in higher-dimensional manifolds. In what follows we denote by $I_{\delta}^k$ the $k$-cube
$$I_{\delta}^k:=[-\delta,\delta]^k.$$

\begin{prop}\label{abodgest} Let $u\in W^{1,p}(I^k_{\delta},N)$ such that $u|_{\partial I^k_{\delta}}\in W^{1,p}(\partial I^k_{\delta},N)$, and set
$$d:=|\int_{\partial I_{\delta}^k}u^*(\alpha)|.$$
Then for any $r>0$, we have the estimate
\begin{equation}
\sigma_{k-1}d^{1+p-k}\leq \lambda(\alpha)^{\frac{p}{k-1}}(r/2)^{p-k}(k-p)[E_p(u,I^k_{\delta})+C(k)r E_p(u,\partial I^k_{\delta})].
\end{equation}
\end{prop}
\begin{proof} We argue as in \cite{ABOvar}. Extend $u$ to a map $\tilde{u}$ on $I^k_{\delta+r}$ by setting 
$$\tilde{u}(x):=u(\delta\cdot x/|x|_{\infty})\text{ when }\delta \leq |x|_{\infty}\leq r,$$
so that
$$E_p(\tilde{u}, I_{\delta+r}^k)\leq E_p(u,I^k_{\delta})+C(k)r E_p(u,\partial I^k_{\delta}).$$
In view of Remark \ref{jprophyps}, we can then apply Proposition \ref{jerrprop1} to the map $\tilde{u}$ on $I^k_{\delta+r}$ to see that
\begin{eqnarray*}
\frac{c(N,\alpha,p)}{k-p} (r/2)^{k-p} d^{1+p-k}&\leq &E_p(\tilde{u}, I_{\delta+r}^k)\\
&\leq & E_p(u,I^k_{\delta})+C(k)r E_p(u,\partial I^k_{\delta}).
\end{eqnarray*}
Recalling that
$$c(N,\alpha,p):=\frac{\sigma_{k-1}}{\lambda(\alpha)^{\frac{p}{k-1}}},$$
the desired estimate follows immediately.
\end{proof}

\subsection{Compactness Results for $T_{\alpha}(u_p)$ as $p\to k$}\label{cpctthm}\hspace{20mm}

\hspace{3mm} Henceforth, let $M^n$ be a closed, oriented Riemannian manifold of dimension $n\geq k$. In this section and the next, we analyze the limiting behavior as $p\to k$ of the homological singularities $T_{\alpha}(u_p)$ for maps $u_p\in W^{1,p}(M,N)$ with energy growth of the form $E_p(u_p)=O(\frac{1}{k-p})$. Our results are inspired in large part by those of \cite{ABOvar} and \cite{JSgl}, concerning the limiting behavior of Jacobian currents for maps of controlled energy growth with respect to functionals of Ginzburg-Landau type.

\hspace{3mm} The starting point for our compactness results is the following proposition--inspired by arguments in \cite{ABOvar}--in which we construct good approximations $\tilde{u}\in \mathcal{E}^p(M,N)$ to given maps $u\in W^{1,p}(M,N)$, such that the mass $T_{\alpha}(\tilde{u})$ is controlled uniformly.

\begin{prop}\label{epapprox} For any $u\in W^{1,p}(M,N)$ with $k-\frac{1}{2}<p<k$ and 
$$E_p(u)\leq \frac{\Lambda}{k-p},$$
there exists a map $\tilde{u}\in \mathcal{E}^p(M,N)$ for which
\begin{equation}\label{eplp}
\|u-\tilde{u}\|_{L^p(M)}^p\leq C(M)(k-p)^{3p-2}\Lambda,
\end{equation}
\begin{equation}\label{epenerg}
E_p(\tilde{u})\leq \frac{C(M)\Lambda}{(k-p)^2},
\end{equation}
and
\begin{equation}\label{epmassbd}
\mathbb{M}(T_{\alpha}(\tilde{u}))\leq C(M,\alpha,\Lambda)
\end{equation}
\end{prop}

\begin{proof} For $\delta\in (0,1)$, choose a cubeulation $h:|K_{\delta}|\to M$ satisfying the conclusions of Lemma \ref{slicelemma}, and let $\tilde{u}_0=u|_{|K_{\delta}^{k-1}|}\circ \Phi_k\circ h^{-1}$. By Lemma \ref{wkapplem}, we then have the estimates
\begin{equation}\label{deltalpest}
\|u-\tilde{u}_0\|_{L^p(M)}^p\leq \delta^p\frac{C(M)}{k-p}\leq E_p(u)\leq \delta^p\frac{C(M)\Lambda}{(k-p)^2}
\end{equation}
and
\begin{equation}\label{genenergest}
E_p(\tilde{u}_0)\leq \frac{C(M)}{k-p}E_p(u) \leq \frac{C(M)\Lambda}{(k-p)^2}.
\end{equation}
Since $p>k-1$, we can then find $f\in Lip(|K_{\delta}^{k-1}|,N)$ homotopic to $u|_{|K_{\delta}^{k-1}|}$ on $|K_{\delta}^{k-1}|$ and arbitrarily close in $W^{1,p}(|K_{\delta}^{k-1}|)$. In particular, we can choose $f$ homotopic to $u|_{|K_{\delta}^{k-1}|}$ such that
$$\tilde{u}:=f\circ \Phi_k\circ h^{-1}\in \mathcal{E}^p(M,N)$$
satisfies (\ref{deltalpest}) and (\ref{genenergest})--modifying the constant $C(M)$ if necessary. Defining $\tilde{u}$ in this way, and taking 
$$\delta=\delta_p:=(k-p)^3,$$
the bounds (\ref{eplp}) and (\ref{epenerg}) follow immediately.

\hspace{3mm} To estimate the mass of $T_{\alpha}(\tilde{u})$, we first appeal to Lemma \ref{epint} to see that 
\begin{equation}\label{twkapp}
T_{\alpha}(\tilde{u})=\Sigma_{\sigma\in K^k\setminus K^{k-1}}\theta(u,\sigma)\cdot [P(\sigma)],
\end{equation}
where
$$|\theta(u,\sigma)|=|\int_{\partial \sigma}\tilde{u}^*(\alpha)|=|\int_{\partial\sigma}u^*(\alpha)|.$$
Now, recalling (\ref{dualcellsest}), we have the volume bound
$$\mathcal{H}^{n-k}(P(\sigma))\leq C(M)\delta^{n-k}$$
for every $k$-cell $\sigma \in K$, so by (\ref{twkapp}), the mass of $T_{\alpha}(\tilde{u})$ is bounded by
\begin{equation}\label{twkmass}
\mathbb{M}(T_{\alpha}(\tilde{u}))\leq \Sigma_{\sigma\in K_{\delta}^k\setminus K_{\delta}^{k-1}}|\theta(u,\sigma)|\delta^{n-k}.
\end{equation}

\hspace{3mm} On the other hand, Proposition \ref{abodgest} (with $r=\delta$) furnishes us with an estimate of the form
\begin{eqnarray*}
|\theta(u,\sigma)|&\leq & C(M,\alpha)\left(\delta^{p-k}(k-p)[E_p(u,\sigma)+\delta E_p(u,\partial\sigma)]\right)^{\frac{1}{1+p-k}}\\
&\leq & C(M,\alpha)\delta^{\frac{p-k}{1+p-k}}(k-p)[E_p(u,\sigma)+\delta E_p(u,\partial\sigma)]\\
&&\cdot\left((k-p)[E_p(u,\sigma)+\delta E_p(u,\partial\sigma)]\right)^{\frac{k-p}{1+p-k}}
\end{eqnarray*}
on every $k$-cell $\sigma$.
Now, since the cubeulation $|K_{\delta}|\to M$ was chosen according to Lemma \ref{slicelemma}, we know that
\begin{equation}\label{skelenergests}
E_p(u,|K^k|)+\delta E_p(u,|K^{k-1}|)\leq C(M)\delta^{k-n}\frac{\Lambda}{k-p},
\end{equation}
from which we immediately obtain the simple-minded estimate
$$(k-p)[E_p(u,\sigma)+\delta E_p(u,\partial\sigma)]\leq C(M)\Lambda \delta^{k-n}$$
for every $k$-cell $\sigma \in K_{\delta}$. Using this to bound the last factor on the right-hand side of the preceding estimate for $|\theta(u,\sigma)|$, we find that
\begin{equation}\label{goodthetaest}
|\theta(u,\sigma)|\leq C(M,\alpha)(k-p)[E_p(u,\sigma)+\delta E_p(u,\partial\sigma)]\cdot [\Lambda \delta^{k-n-1}]^{\frac{k-p}{1+p-k}}.
\end{equation}

\hspace{3mm} On the other hand, summing $E_p(u,\sigma)+\delta E_p(u,\partial\sigma)$ over all $k$-cells $\sigma\in K_{\delta}$, we also have the bound
\begin{eqnarray*}
\Sigma_{\sigma}[E_p(u,\sigma)+\delta E_p(u,\partial\sigma)]&\leq& C(M)[E_p(u,|K^k|)+\delta E_p(u,|K^{k-1}|)]\\
&\leq & C(M)\delta^{k-n}\frac{\Lambda}{k-p}.
\end{eqnarray*}
In particular, summing (\ref{goodthetaest}) over all $k$-cells and employing the estimate above, we find that
$$\Sigma_{\sigma}|\theta(u,\sigma)| \leq C'(M,\alpha)\cdot \Lambda \delta^{k-n}\cdot [\Lambda \delta^{k-n-1}]^{\frac{k-p}{1+p-k}}.$$
Recalling now the bound (\ref{twkmass}) for $\mathbb{M}(T_{\alpha}(\tilde{u}))$, we deduce that
$$\mathbb{M}(T_{\alpha}(\tilde{u}))\leq C(M,\alpha)\Lambda \cdot [\Lambda \delta^{k-1-n}]^{\frac{k-p}{1+p-k}}.$$
Since we've set $\delta=\delta_p=(k-p)^3$, we check directly that
$$\sup_{k-\frac{1}{2}<p<k}[\Lambda \delta_p^{k-1-n}]^{\frac{k-p}{1+p-k}}<\infty,$$
and the desired mass bound (\ref{epmassbd}) follows.
\end{proof}

\hspace{3mm} Given a family of maps $(k-1,k)\ni p\mapsto u_p\in W^{1,p}(M,N)$ with $E_p(u_p)=O(\frac{1}{k-p})$, Proposition \ref{epapprox} gives us an associated family of integral $(n-k)$-cycles $T_{\alpha}(\tilde{u}_p)$ with uniform mass bounds. By showing that 
$$T_{\alpha}(u_p)-T_{\alpha}(\tilde{u}_p)\to 0$$
in $(C^1)^*$ as $p\to k$, and applying the Federer-Fleming compactness theorem to the cycles $T_{\alpha}(\tilde{u}_p)$, we arrive at the following preliminary compactness result.

\begin{cor}\label{cpct1} Let $p_j\in (k-1,k)$ be a sequence with $p_j\to k$, and let $u_j\in W^{1,p_j}(M,N)$ be a sequence of maps satisfying
\begin{equation}\label{epgrowth}
\limsup_{j\to\infty}(k-p_j)E_{p_j}(u_j)\leq \Lambda<\infty.
\end{equation}
Then there exists a subsequence (unrelabelled) $p_j\to k$ such that $T_{\alpha}(u_j)$ converges in $(C^1)^*$ to an integer rectifiable cycle $T\in \mathcal{Z}_{n-k}(M;\mathbb{Z})$ of finite mass.
\end{cor}
\begin{proof} To each map $u_j$, by Proposition \ref{epapprox}, we can associate a map $\tilde{u}_j\in \mathcal{E}^p(M,N)$ for which
$$\|u_j-\tilde{u_j}\|_{L^p}^p\leq C\Lambda (k-p_j)^{3p_j-2},$$
$$E_p(\tilde{u}_j)\leq \frac{C}{(k-p_j)^2}\Lambda,$$
and
$$\mathbb{M}(T_{\alpha}(\tilde{u}_j))\leq C(M,\alpha,\Lambda).$$
On the other hand, by Lemma \ref{sdecomplem}, we know that
$$T_{\alpha}(u_j)-T_{\alpha}(\tilde{u}_j)=\partial S_{\alpha}(u_j,\tilde{u}_j),$$
where
\begin{eqnarray*}
\mathbb{M}(S_{\alpha}(u_j,\tilde{u}_j))&\leq &C(\alpha)[E_{p_j}(u_j)^{\frac{k-1}{p_j}}+E_{p_j}(\tilde{u}_j)^{\frac{k-1}{p_j}}]\|u_j-\tilde{u}_j\|_{L^{p_j}}^{1+p_j-k}\\
&\leq &C [\Lambda/(k-p_j)^2]^{\frac{k-1}{p_j}}\cdot (C\Lambda (k-p_j)^{3p_j-2})^{\frac{1+p_j-k}{p_j}}\\
&\leq & C(M,\alpha,\Lambda)(k-p_j)^{3(p_j+1-k)-2},
\end{eqnarray*}
so in particular,
$$\lim_{j\to\infty}\mathbb{M}(S_{\alpha}(u_j,\tilde{u}_j))=0.$$

\hspace{3mm} Since the currents $T_{\alpha}(\tilde{u}_j)$ are integral cycles with uniformly bounded mass, it follows from the Federer-Fleming compactness theorem (see \cite{Fed}, Theorem 4.2.17) that--after passing to a subsequence--there exists an integral cycle $T\in \mathcal{Z}_{n-k}(M;\mathbb{Z})$ and a sequence of integer-rectifiable $(n+1-k)$-currents $\Gamma_j\in \mathcal{I}_{n+1-k}(M;\mathbb{Z})$ such that
$$\lim_{j\to\infty}\mathbb{M}(\Gamma_j)=0$$
and
$$\partial\Gamma_j=T_{\alpha}(\tilde{u}_j)-T.$$

\hspace{3mm} Putting all this together, we see that
$$T_{\alpha}(u_j)-T=\partial (S_{\alpha}(u_j,\tilde{u}_j)+\Gamma_j)$$
and 
$$\mathbb{M}(S_{\alpha}(u_j,\tilde{u}_j)+\Gamma_j)\to 0,$$
from which it clearly follows that $T_{\alpha}(u_j)-T\to 0$ in $(C^1)^*$.
\end{proof}

\begin{remark}\label{wklkrk} For a simple consequence of Corollary \ref{cpct1}, consider a map $u\in W^{1,p}(M,N)$ for which $|du|\in L^{k,\infty}(M)$--that is, for which
\begin{equation}
\|du\|_{L^{k,\infty}}^k:=\sup\{t^kVol(\{|du|>t\})\mid t\in (0,\infty)\}<\infty
\end{equation}
--and note that for $p<k$, we have the straightforward $L^p$ estimate
\begin{eqnarray*}
\int_M|du|^p&=&\int_0^{\infty}pt^{p-1}Vol(\{|du|>t\})dt\\
&\leq &\int_0^1pt^{p-1}Vol(M)dt+\int_1^{\infty}pt^{p-1}Vol(\{|du|>t\})dt\\
&=&Vol(M)+\|du\|_{L^{k,\infty}}^k\frac{p}{k-p}.
\end{eqnarray*}
In particular, the hypotheses of Corollary \ref{cpct1} hold with $u_j=u$, so we see that $T_{\alpha}(u)$ must be an integral cycle.
\end{remark}

\begin{subsection}{Sharp Mass Bounds for the Limiting Current}\hspace{20mm}

\hspace{3mm} Our goal in this section is to establish a sharp upper bound for the mass of the limiting current in Corollary \ref{cpct1}; namely, we prove the following proposition.

\begin{prop}\label{massbd} For a sequence $p_j\in (k-1,k)$ with $\lim_{j\to\infty}p_j=k$ and a sequence of maps $u_j\in W^{1,p_j}(M,N)$ satisfying
$$\limsup_{j\to\infty}(k-p_j)E_{p_j}(u_j)\leq \Lambda$$
and
\begin{equation}
\lim_{j\to\infty}T_{\alpha}(u_j)=T
\end{equation}
in $(C^1)^*$, the limit current $T$ satisfies
\begin{equation}
\sigma_{k-1}\mathbb{M}(T)\leq \lambda(\alpha)^{\frac{k}{k-1}}\Lambda.
\end{equation}
\end{prop}

\hspace{3mm} To prove Proposition \ref{massbd}, we continue to model our arguments on those of (\cite{ABOvar}, Section 3), proving first the following lemma for maps from the Euclidean unit ball $B_1^n(0)$.

\begin{lem}\label{aboballlem} Let $p_j\in (k-1,k)$ be a sequence with $\lim_{j\to\infty}p_j=k$, and let $u_j\in W^{1,p_j}(B_1^n(0),N)$ be a family of maps for which
\begin{equation}\label{usualepbd}
\limsup_{j\to\infty}(k-p_j)E_{p_j}(u_j,B_1^n(0))\leq \Lambda<\infty.
\end{equation}
Then for any simple unit $(n-k)$-covector $\beta\in \bigwedge^{n-k}(\mathbb{R}^n)$ and $\varphi\in C_c^{\infty}(B_1^n)$, we have the estimate
\begin{equation}\label{locsimpest}
\sigma_{k-1}\limsup_{j\to\infty}|\langle T_{\alpha}(u_j),\varphi\cdot \beta\rangle|\leq \lambda(\alpha)^{\frac{k}{k-1}}\Lambda\|\varphi\|_{L^{\infty}}.
\end{equation}
\end{lem}

\begin{proof} After a rotation, it is enough to prove (\ref{locsimpest}) in the case 
$$\beta=dx^1\wedge\cdots\wedge dx^{n-k}.$$
Following the notation of \cite{ABOvar}, for $a\in \mathbb{R}^n$ and $\delta>0$, let $\mathcal{G}(\delta,a)$ denote the grid
$$\mathcal{G}(\delta,a):=a+\delta\cdot \mathbb{Z}^n,$$
and let $R_j(\delta,a)$ denote the $j$-skeleton of the associated $n$-dimensional cubical complex for which $\mathcal{G}(\delta,a)$ gives the vertices. Denote by $\tilde{R}_k(\delta,a)$ the component
$$\tilde{R}_k(\delta,a):=a+(\delta \mathbb{Z}^{n-k}\times \mathbb{R}^k)$$
of $R_k(\delta,a)$ parallel to $\{0\}\times \mathbb{R}^k$. As in Lemma 3.11 of \cite{ABOvar}, a simple Fubini argument shows that for $u\in W^{1,p}(M,N)$ and $\eta>0$, we can find $a(u,\delta,\eta)\in \mathbb{R}^n$ such that
\begin{equation}\label{coarseslicebd}
\int_{R_j(\delta,a)\cap B_1}|du|^p\leq \frac{C}{\eta}\delta^{j-n}\int_{B_1^n}|du|^p
\end{equation}
for all $0\leq j\leq n$ and
\begin{equation}\label{sharpslicebd}
\int_{\tilde{R}_k(\delta,a)\cap B_1}|du|^p\leq (1+\eta)\delta^{k-n}\int_{B_1^n}|du|^p.
\end{equation}

\hspace{3mm} Now, fix some arbitrary $\varphi\in C_c^{\infty}(B_1^n)$ and $\eta>0$, and consider a family of maps $u_j\in W^{1,p_j}(B_1,N)$ satisfying (\ref{usualepbd}). As in the proof of Proposition \ref{epapprox}, we let
$$\delta_j:=(k-p_j)^3,$$
and let $\tilde{u}_j=u_j\circ \Phi_k$ with respect to the cubical complex associated to the grid $\mathcal{G}(\delta_j,a_j(\eta))$--where $a_j(\eta)$ is chosen to satisfy (\ref{coarseslicebd}) and (\ref{sharpslicebd}) with respect to $u_j$. Of course, $\tilde{u}_j$ is only well-defined on those $n$-cells strictly contained in $B^n_1(0)$, but since $\varphi$ is supported in the interior of $B_1$ and $\lim_{j\to\infty}\delta_j=0$, we see that $\tilde{u}_j$ is defined on $spt(\varphi)$ for $j$ sufficiently large.

\hspace{3mm} Setting $\zeta=\varphi dx^1\wedge \cdots\wedge dx^{n-k}$, we can then argue as in the proof of Corollary \ref{cpct1} to see that, for $j$ sufficiently large,
\begin{eqnarray*}
|\langle T_{\alpha}(u_j)-T_{\alpha}(\tilde{u_j}),\zeta\rangle|&=&|\langle S_{\alpha}(u_j,\tilde{u}_j),d\zeta\rangle|\\
&\leq & C(n,\alpha)\eta^{-1}\Lambda (k-p_j)\|d\zeta\|_{L^{\infty}}.
\end{eqnarray*}
In particular, it follows that
\begin{equation}\label{tudiffs}
\lim_{j\to\infty}|\langle T_{\alpha}(u_j)-T_{\alpha}(\tilde{u}_j),\zeta\rangle|=0.
\end{equation}

\hspace{3mm} On the other hand, by Proposition \ref{epint}, we know that
$$\langle T_{\alpha}(\tilde{u}_j),\zeta\rangle=\Sigma_{\sigma\subset\tilde{R}_k(\delta_j,a_j)\cap B_1} \theta(u_j,\sigma)\int_{P(\sigma)}\varphi,$$
where the sum is over all $k$-cells $\sigma\cong [0,\delta]^k$ contained in $\tilde{R}_k(\delta_j,a_j)\cap B_1$, and
$$\theta(u_j,\sigma)=\pm \int_{\partial\sigma}u_j^*(\alpha).$$
In this Euclidean setting, the component $P(\sigma)$ of the dual $(n-k)$-skeleton intersecting $\sigma$ is given by a single $(n-k)$-cell isometric to $[0,\delta]^{n-k}$, and as a consequence, we see that
\begin{equation}\label{tzetapair}
|\langle T_{\alpha}(\tilde{u}_j),\zeta\rangle|\leq \Sigma_{\sigma\subset \tilde{R}_k(\delta_j,a_j)\cap B_1}|\theta(u_j,\sigma)|\delta^{n-k}\|\varphi\|_{L^{\infty}}.
\end{equation}

\hspace{3mm} To estimate the coefficients $|\theta(u_j,\sigma)|$, we first appeal to Proposition \ref{abodgest} and (\ref{coarseslicebd}) to get the crude estimate
\begin{eqnarray*}
|\theta(u,\sigma)|^{1+p_j-k}&\leq & C(k,\alpha)\delta_j^{p_j-k}(k-p_j)[E_{p_j}(u_j,\sigma)+\delta_jE_{p_j}(u_j,\partial\sigma)]\\
&\leq & C(k,n,\alpha,\eta)\delta_j^{p_j-n}\Lambda.
\end{eqnarray*}
In particular, setting 
$$c_j:=[C(k,n,\alpha,\eta)\delta_j^{p_j-n}\Lambda]^{\frac{k-p_j}{1+p_j-k}},$$
we have the bound
\begin{equation}\label{thetafactorbd}
|\theta(u_j,\sigma)|^{k-p_j}\leq c_j,
\end{equation}
and recalling that $\delta_j=(k-p_j)^3$, we observe that
$$\lim_{j\to\infty}c_j=1.$$ 

\hspace{3mm} For a finer estimate, we appeal again to Proposition \ref{abodgest} to see that, for any $r>0$ and any $k$-cell $\sigma$,
$$\sigma_{k-1}|\theta(u,\sigma)|^{1+p_j-k}\leq \lambda(\alpha)^{\frac{p_j}{k-1}}(\delta_jr)^{p_j-k}(k-p_j)[E_{p_j}(u_j,\sigma)+C(k)r\delta_j E_p(u_j,\partial\sigma)].$$
Multiplying both sides above by $|\theta(u_j,\sigma)|^{k-p_j}$ and appealing to (\ref{thetafactorbd}), we then arrive at the bound
$$\sigma_{k-1}|\theta(u_j,\sigma)|\leq c_j\lambda(\alpha)^{\frac{p_j}{k-1}}(\delta_jr)^{p_j-k}(k-p_j)[E_{p_j}(u,\sigma)+C(k)r\delta_jE_{p_j}(u,\sigma)].$$
Summing over $k$-cells $\sigma\subset \tilde{R}_k(\delta_j,a_j)$, and appealing to (\ref{coarseslicebd}) and (\ref{sharpslicebd}), we find that
\begin{eqnarray*}
\sigma_{k-1}\cdot \Sigma_{\sigma\subset \tilde{R}_k(\delta_j,a_j)}|\theta(u_j,\sigma)|&\leq &c_j\lambda(\alpha)^{\frac{p_j}{k-1}}(\delta_jr)^{p_j-k}(k-p_j)\\
&&\cdot\left(\int_{\tilde{R}_k\cap B_1}|du_j|^{p_j}+C(k)r\delta \int_{R_{k-1}\cap B_1}|du_j|^{p_j}\right)\\
&\leq &c_j\lambda(\alpha)^{\frac{p_j}{k-1}}(\delta_jr)^{p_j-k}(k-p_j)\\
&&\cdot [(1+\eta)+\frac{C(n,k)}{\eta} r]\delta_j^{k-n}E_{p_j}(u_j,B_1).
\end{eqnarray*}

\hspace{3mm} Choosing $r=\eta^2$ above, and returning to (\ref{tzetapair}), we arrive at the estimate
\begin{equation}
\sigma_{k-1}|\langle T_{\alpha}(\tilde{u}_j),\zeta\rangle|\leq c_j\lambda(\alpha)^{\frac{p_j}{k-1}}(\delta_j \eta^2)^{p_j-k} [1+C'(n,k)\eta]\Lambda\|\varphi\|_{L^{\infty}}.
\end{equation}
Now, since $\lim_{j\to\infty}c_j=1$, and likewise $\lim_{j\to\infty}(\delta_j\eta^2)^{p_j-k}=1$, we deduce that
\begin{equation}
\limsup_{j\to\infty}\sigma_{k-1}|\langle T_{\alpha}(\tilde{u}_j),\zeta\rangle|\leq \lambda(\alpha)^{\frac{k}{k-1}}[1+C\eta]\Lambda \|\varphi\|_{L^{\infty}}.
\end{equation}
By (\ref{tudiffs}), this is equivalent to the statement that
\begin{equation}
\limsup_{j\to\infty}\sigma_{k-1}|\langle T_{\alpha}(u_j),\zeta\rangle|\leq \lambda(\alpha)^{\frac{k}{k-1}}[1+C\eta]\Lambda \|\varphi\|_{L^{\infty}};
\end{equation}
finally, taking $\eta\to 0$, we arrive at the desired estimate.
\end{proof}

\hspace{3mm} With Lemma \ref{aboballlem} in hand, we can now prove Proposition \ref{massbd} via a blow-up argument.

\begin{proof}{(Proof of Proposition \ref{massbd})}

\hspace{3mm} Let $u_j\in W^{1,p_j}(M,N)$ be a sequence of maps as in Proposition \ref{cpct1}, for which
$$\limsup_{j\to\infty}(k-p_j)E_{p_j}(u_j)\leq \Lambda$$
and
$$\lim_{j\to\infty}T_{\alpha}(u_j)=T\in \mathcal{Z}_{n-k}(M;\mathbb{Z}).$$
Passing to a further subsequence, we can also assume that the normalized energy measures
$$\mu_j:=(k-p_j)|du_j|^{p_j}dv_g$$
converge weakly in $(C^0)^*$ to a limiting Radon measure $\mu$ satisfying
$$\mu(M)\leq \Lambda.$$
Denote by $|T|$ the weight measure associated to the current $T$. By standard results on derivates of Radon measures (see, e.g., \cite{Sim}, Section 4 or \cite{Fed}, Section 2.9), the quantity
$$D_{\mu}|T|(x):=\lim_{r\to 0}\frac{|T|(B_r(x))}{\mu(B_r(x))}$$
is well-defined for $|T|$-a.e. $x\in M$, and to establish the desired mass bound for $T$, it will suffice to show that
\begin{equation}\label{derivbd}
D_{\mu}|T|(x)\leq \sigma_{k-1}^{-1}\lambda(\alpha)^{\frac{k}{k-1}}\text{ for }|T|-a.e.\text{ }x\in M.
\end{equation}

\hspace{3mm} Now, on a small geodesic ball $B_r(x)\subset M$, denote by 
$$\Phi_{x,r}: B_r(x)\to B_1^n(0)\subset T_xM$$
the dilation map
$$\Phi_{x,r}(y):=\frac{1}{r}\exp_x^{-1}(y),$$
and set
$$\mu_{x,r}:=(\Phi_{x,r})_{\#}\mu,\text{\hspace{3mm} }T_{x,r}:=(\Phi_{x,r})_{\#}T.$$
Since $T$ is integer rectifiable, for $|T|$-almost every $x\in M$, the currents $T_{x,r}$ converge weakly 
\begin{equation}\label{goodblowup}
T_{x,r}\rightharpoonup \theta(x) [P]\in\mathcal{D}_{n-k}(B_1^n),
\end{equation}
to an oriented $(n-k)$-plane $P$ in $\mathbb{R}^n$ with multiplicity 
$$\theta(x):=\lim_{r\to 0}\frac{|T|(B_r(x))}{\omega_{n-k}r^{n-k}},$$
where $\omega_{n-k}:=\mathcal{L}^{n-k}(B_1^{n-k}(0))$ (see, for instance, \cite{Sim}, Section 32).

\hspace{3mm} Now, let $x\in M$ be a point at which $D_{\mu}|T|(x)$ is defined and (\ref{goodblowup}) holds, and observe that the density $\Theta_{n-k}(\mu,x)$ is then well-defined (though possibly infinite), as
\begin{eqnarray*}
\Theta_{n-k}(\mu,x)&=&\lim_{r\to 0}\frac{\mu(B_r(x))}{\omega_{n-k}r^{n-k}}\\
&=&\theta(x)\lim_{r\to 0}\frac{\mu(B_r(x))}{|T|(B_r(x))}\\
&=&\theta(x)\frac{1}{D_{\mu}|T|(x)}.
\end{eqnarray*}
In particular, to prove (\ref{derivbd}), we just need to show that
\begin{equation}\label{denscomp}
\sigma_{k-1}\theta(x)\leq \lambda(\alpha)^{\frac{k}{k-1}}\Theta_{n-k}(\mu,x).
\end{equation}

\hspace{3mm} If $\Theta_{n-k}(\mu,x)=\infty$, then (\ref{denscomp}) holds trivially, so assume that 
$$\Theta_{n-k}(\mu,x)<\infty,$$
and consider a sequence $r_{\ell}\to 0$ for which
$$\mu(\partial B_{r_{\ell}}(x))=0.$$
For each $r_{\ell}$, we then have
$$\mu(B_{r_{\ell}}(x))=\lim_{j\to\infty}(k-p_j)\int_{B_{r_{\ell}}(x)}|du_j|^{p_j},$$
and consequently
\begin{equation}\label{ballmeasconv}
r_{\ell}^{k-n}\mu(B_{r_{\ell}}(x))=\lim_{j\to\infty}r_{\ell}^{p_j-n}(k-p_j)\int_{B_{r_{\ell}}(x)}|du_j|^{p_j}.
\end{equation}
Moreover, since the convergence $T_{\alpha}(u_j)\to T$ established in Proposition \ref{cpct1} is convergence in the $(C^1)^*$ norm, we also see that
\begin{equation}\label{balltconv}
\lim_{j\to\infty}\|T-T_{\alpha}(u_j)\|_{(C_c^1(B_{r_{\ell}}(x))^*}=0.
\end{equation}

\hspace{3mm} It follows from (\ref{ballmeasconv}) and (\ref{balltconv}) that for each $r_{\ell}$, we can select $p_{\ell}=p_{j_{\ell}}$ and $u_{\ell}=u_{j_{\ell}}$ such that
\begin{equation}
|\frac{\mu(B_{r_{\ell}}(x))}{r_{\ell}^{n-k}}-r_{\ell}^{p_{\ell}-n}(k-p_{\ell})\int_{B_{r_{\ell}}(x)}|du_{\ell}|^{p_{\ell}}|<\frac{1}{\ell}
\end{equation}
and 
\begin{equation}
r_{\ell}^{k-n-1}\|T-T_{\alpha}(u_{\ell})\|_{(C_c^1(B_{r_{\ell}}(x)))^*}<\frac{1}{\ell}.
\end{equation}
Defining the maps $v_{\ell}\in W^{1,p_{\ell}}(B_1^n(0),N)$ by 
$$v_{\ell}:=u\circ \Phi_{x,r_{\ell}}^{-1},$$
we then see that
\begin{equation}\label{venerglim}
\lim_{\ell\to\infty}(k-p_{\ell})E_{p_{\ell}}(v_{\ell},B_1)=\omega_{n-k}\Theta_{n-k}(\mu,x)
\end{equation}
and
\begin{equation}\label{vtlim}
\lim_{\ell\to \infty}\langle T_{\alpha}(v_{\ell}),\zeta\rangle=\lim_{r\to 0}\langle T_{x,r},\zeta\rangle=\theta(x)\langle [P],\zeta\rangle
\end{equation}
for all $\zeta\in \Omega_c^{n-k}(B_1^n(0))$.

\hspace{3mm} Now, applying Lemma \ref{aboballlem} to the maps $v_{\ell}$ and the simple unit $(n-k)$-covector $\beta$ orienting $[P]$, we deduce from (\ref{venerglim}) and (\ref{vtlim}) that
\begin{equation}\label{phipair}
\sigma_{k-1}\theta(x)\int_{P}\varphi=\sigma_{k-1}\theta(x)\langle [P],\varphi \cdot \beta\rangle \leq \lambda(\alpha)^{\frac{k}{k-1}}\omega_{n-k}\Theta_{n-k}(\mu,x)\|\varphi\|_{L^{\infty}}
\end{equation}
for any $\varphi\in C_c^{\infty}(B_1)$. Finally, applying (\ref{phipair}) to a reasonable approximation $\varphi_j\to\chi_{B_1(0)}$ with $\|\varphi_j\|_{L^{\infty}}\leq 1$, we obtain in the limit
$$\sigma_{n-k}\theta(x)\omega_{n-k}\leq \lambda(\alpha)^{\frac{k}{k-1}}\omega_{n-k}\Theta_{n-k}(\mu,x).$$
Dividing through by $\omega_{n-k}$ gives precisely (\ref{denscomp}), and the proposition follows.
\end{proof}

\hspace{3mm} Finally, combining the results of Proposition \ref{epapprox}, Corollary \ref{cpct1}, and Proposition \ref{massbd} above, together with a simple contradiction argument, we obtain the following strong version of the compactness result.

\begin{thm}\label{bigcpctthm} For any $\Lambda<\infty$ and $\eta>0$, there exists $q(M,\alpha,\Lambda,\eta)\in (k-1,k)$ such that if $p\in (q,k)$, and $u\in W^{1,p}(M,N)$ satisfies
$$(k-p)E_p(u)\leq \Lambda,$$
then there exists a map $\tilde{u}\in \mathcal{E}^p(M,N)$ satisfying
\begin{equation}
E_p(\tilde{u})\leq C(M)\frac{\Lambda}{(k-p)^2},
\end{equation}
\begin{equation}
\|u-\tilde{u}\|^p_{L^p(M)}\leq (k-p)^{3p-2}C(M)\Lambda,
\end{equation}
and an integral $(n-k)$-cycle $T\in \mathcal{Z}_{n-k}(M;\mathbb{Z})$ and integral $(n+1-k)$-current $\Gamma\in \mathcal{I}_{n+1-k}(M;\mathbb{Z})$ such that
\begin{equation}
T_{\alpha}(\tilde{u})-T=\partial\Gamma,
\end{equation}
\begin{equation}
\sigma_{k-1}\mathbb{M}(T)\leq \lambda(\alpha)^{\frac{k}{k-1}}\Lambda,
\end{equation}
and
\begin{equation}
\mathbb{M}(\Gamma)<\eta.
\end{equation}
\end{thm}

\end{subsection}

\section{The Lower Bounds for $\gamma_p^*(u,v)$}\label{lbdssec}

\subsection{The Almgren Isomorphism and Min-Max Widths}\label{almsubsec}\hspace{30mm}

\hspace{3mm} Before proving Theorem \ref{lbdsthm}, we recall some basic facts about the map $\pi_1(\mathcal{Z}_m(M;\mathbb{Z}),0)\to H_{m+1}(M;\mathbb{Z})$ constructed by Almgren in \cite{A1}, and make precise the definition of the min-max widths that we use to obtain lower bounds. 

\hspace{3mm} As in \cite{A1}, we topologize the space $\mathcal{Z}_m(M;\mathbb{Z})$ of integral $m$-cycles via the flat norm 
$$\mathbb{F}(T):=\inf\{\mathbb{M}(T')+\mathbb{M}(S)\mid T=T'+\partial S,\text{ }T'\in \mathcal{Z}_m(M;\mathbb{Z}),\text{ }S\in \mathcal{I}_{m+1}(M;\mathbb{Z})\}.$$
In his dissertation \cite{A1}, Almgren exhibited an isomorphism 
$$\pi_{\ell}(\mathcal{Z}_m(M;\mathbb{Z}),0)\cong H_{\ell+m}(M;\mathbb{Z})$$
between the homotopy groups of $\mathcal{Z}_m$ and the homology groups of $M$, which he later employed in \cite{A2} for the purpose of constructing minimal submanifolds via min-max methods. Recent years have seen a tremendous renewal of interest in the topology of spaces of cycles and related min-max constructions of minimal submanifolds (particularly in codimension one, where Pitts's work provides a powerful regularity theory \cite{Pi}): we refer the interested reader to \cite{Guth}, \cite{MN}, \cite{LMN}, \cite{IMN}, \cite{MNS}, \cite{Song}, and references therein for some recent developments. 

\hspace{3mm} In the case $\ell=1$ of interest for our present purposes, the map
$$\Psi: \pi_1(\mathcal{Z}_m(M;\mathbb{Z}),0)\to H_{m+1}(M;\mathbb{Z})$$
is fairly simple to describe. First, by two applications of the Federer-Fleming isoperimetric inequality on manifolds (\cite{Fed}, Theorem 4.4.2), there exists a constant $\epsilon(M)>0$ such that if 
$$R\in \mathcal{Z}_{m+1}(M;\mathbb{Z})\text{ with }\mathbb{M}(R)<\epsilon,$$
then $R=\partial \Omega$ for some $\Omega\in \mathcal{I}_{m+2}(M;\mathbb{Z})$, and there exists also $\delta>0$ such that if 
$$T\in \mathcal{Z}_m(M;\mathbb{Z})\text{ with }\mathbb{F}(T)<\delta,$$
then $T=\partial S$ for some $S\in \mathcal{I}_{m+1}(M;\mathbb{Z})$ such that
$$\mathbb{M}(S)<\frac{\epsilon}{2}.$$

\hspace{3mm} As a consequence, if $T_0,T_1,\ldots,T_r$ is a finite sequence in $\mathcal{Z}_m(M;\mathbb{Z})$ with $T_0=T_r=0$ and 
\begin{equation}\label{flatfine}
\mathbb{F}(T_i-T_{i-1})<\delta
\end{equation}
for every $i=1,\ldots,r$ (as can be obtained, for instance, by sampling points from a loop in $\pi_1(\mathcal{Z}_m(M;\mathbb{Z}),0)$), then we can find $S_i\in \mathcal{I}_{m+1}$ for which 
\begin{equation}\label{fillins}
\partial S_i=T_i-T_{i-1}\text{ and }\mathbb{M}(S_i)<\frac{\epsilon}{2}.
\end{equation}
The sum $\overline{S}=\Sigma_{i=1}^rS_i$ then defines a cycle in $\mathcal{Z}_{m+1}(M;\mathbb{Z})$, and for any other choices $S_i'\in \mathcal{I}_{m+1}$ satisfying (\ref{fillins}), the differences $R_i:=S_i'-S_i$ are $(m+1)$-cycles of mass $\mathbb{M}(R_i)<\epsilon$, so that
$$S_i'=S_i+\partial\Omega_i$$
for some $\Omega_i\in \mathcal{I}_{m+2}(M;\mathbb{Z})$. In particular, it follows that the homology class $[\overline{S}]\in H_{m+1}(M;\mathbb{Z})$ of $\overline{S}$ is independent of the choice of $S_i$ in (\ref{fillins}). In a similar way (taking $\delta$ in (\ref{flatfine}) smaller, if necessary), it is shown in \cite{A1} that the homology class $[\overline{S}]$ produced in this way remains constant over sequences $\{T_i\}$, $\{T_i'\}$ which are close in an appropriate sense, and it is this observation which accounts for the well-definedness of the map $\Psi: \pi_1(\mathcal{Z}_m(M;\mathbb{Z}),0)\to H_{m+1}(M;\mathbb{Z})$.

\hspace{3mm} Given a homotopy class $\Pi\in \pi_1(\mathcal{Z}_m(M;\mathbb{Z}),0)$, for the purposes of intuition, the min-max width ${\bf L}(\Pi)$ can be identified with the quantity
\begin{equation}\label{naivewidth}
\inf_{F\in \Pi}\sup_{y\in S^m}\mathbb{M}(F(y)),
\end{equation}
giving the infimum over all families $F\in \Pi$ of the maximal mass attained by a cycle in the family $F$. For the purposes of this paper, however, it is convenient to define the min-max widths in terms of finite sequences $\{T_j\}$ in $\mathcal{Z}_m$ for which adjacent cycles are close in flat norm. The interested reader can compare the definition given below with those of \cite{A2}, \cite{Pi} (which require fineness in stronger norms), referring to interpolation procedures like those described in (\cite{Gu}, Section 8). 

\hspace{3mm} For $\delta>0$, we denote by $\mathscr{S}_{m,\delta}(M)$ the collection of all finite sequences $\{T_i\}_{i=0}^r$ of integral $m$-cycles $T_i\in \mathcal{Z}_m(M;\mathbb{Z})$ for which
$$T_0=T_r=0\text{ and }\mathbb{F}(T_{i-1}-T_i)<\delta\text{ for every }i=1,\ldots,r.$$
By the discussion in the preceding paragraphs (or see again \cite{A1}, \cite{A2} Chapter 13), there are constants $\epsilon_0(M)>0$ and $\delta_0(M)>0$ such that for $\delta<\delta_0$, the map
$$\Psi: \mathscr{S}_{m,\delta}(M)\to H_{m+1}(M;\mathbb{Z})$$
given by
$$\Psi(\{T_i\})=[\Sigma_{i=1}^rS_i]$$
for some $S_i\in \mathcal{I}_{m+1}(M;\mathbb{Z})$ satisfying
\begin{equation}\label{fillinconds}
\partial S_i=T_i-T_{i-1}\text{ and }\mathbb{M}(S_i)<\frac{\epsilon_0}{2}
\end{equation}
is well-defined, independent of the choice of $\{S_i\}$ satisfying (\ref{fillinconds}).

\hspace{3mm} Given a homology class $\xi \in H_{m+1}(M;\mathbb{Z})$ and $\delta<\delta_0$, we then set
\begin{equation}\label{lmdeltadef}
{\bf L}_{m,\delta}(\xi):=\inf\{\max_{0\leq i\leq r}\mathbb{M}(T_i)\mid \{T_i\}_0^r\in \mathscr{S}_{m,\delta}(M),\text{ }\Psi(\{T_i\})=\xi\},
\end{equation}
and define
\begin{equation}\label{lmdef}
{\bf L}_m(\xi):=\lim_{\delta\to 0}{\bf L}_{m,\delta}(\xi)=\sup_{\delta>0}{\bf L}_{m,\delta}(\xi).
\end{equation}
By another simple application of the isoperimetric inequality, one finds that
$$\inf_{\xi\neq 0}{\bf L}_m(\xi)>0;$$ 
to see this, note that there exists a constant $\eta_1(M)>0$ such that for every $T\in \mathcal{Z}_m(M;\mathbb{Z})$ with $\mathbb{M}(T)<\eta_1$, there is some $R\in \mathcal{I}_{m+1}(M;\mathbb{Z})$ satisfying
\begin{equation}\label{smallmasscase}
T=\partial R\text{ and }\mathbb{M}(R)<\frac{\epsilon_0}{4}.
\end{equation}
In particular, if we have $\{T_i\}\in \mathscr{S}_{m,\delta}(M)$ with $\max_i \mathbb{M}(T_i)<\eta_1$, then for each $i=1,\ldots,r-1$ we can choose $R_i\in \mathcal{I}_{m+1}(M;\mathbb{Z})$ satisfying (\ref{smallmasscase}) with respect to $T_i$, and setting $R_0=R_r=0$, we see that the currents
$$S_i:=R_i-R_{i-1}$$
satisfy (\ref{fillinconds}). But evidently $\Sigma_{i=1}^rS_i=0$, so it follows that $\Psi(\{T_i\})=0.$

\hspace{3mm} For the families of cycles $\{T_i\}\in \mathscr{S}_{n-k,\delta}(M)$ arising in the proof of Theorem \ref{lbdsthm}, we can determine the associated homology class $\Psi(\{T_i\})$ only at the level of real homology. For any real homology class $\overline{\xi} \in H_{m+1}(M;\mathbb{R})$ containing an integral representative $S\in \mathcal{Z}_{m+1}(M;\mathbb{Z}),$ we therefore define the real-homological widths
\begin{equation}
{\bf L}_{m,\mathbb{R}}(\overline{\xi}):=\min\{{\bf L}_m(\xi)\mid \xi\in H_{m+1}(M;\mathbb{Z}),\text{ }\xi\equiv \overline{\xi}\text{ in }H_{m+1}(M;\mathbb{R})\}.
\end{equation}
Equivalently, we can set
\begin{equation}
{\bf L}_{m,\mathbb{R},\delta}(\overline{\xi}):=\min\{{\bf L}_{m,\delta}(\xi)\mid \xi\equiv \overline{\xi}\text{ in }H_{m+1}(M;\mathbb{R})\},
\end{equation}
and define ${\bf L}_{m,\mathbb{R}}(\overline{\xi})$ by
$${\bf L}_{m,\mathbb{R}}(\overline{\xi}):=\lim_{\delta\to 0}{\bf L}_{m,\mathbb{R},\delta}(\overline{\xi})=\sup_{\delta>0}{\bf L}_{m,\mathbb{R},\delta}(\overline{\xi}).$$

\hspace{3mm} The need to work with real homology in the proof of Theorem \ref{lbdsthm} is due in part to the fact that the currents $S_i\in \mathcal{D}_{n+1-k}(M)$ that we use to connect adjacent $(n-k)$-cycles $T_i-T_{i-1}=\partial S_i$ are not integer-rectifiable. However, from the results of Section \ref{homints}, we will see that they have the form $S_i=\Gamma_i+\partial R_i$, where $\Gamma_i\in \mathcal{I}_{n+1-k}(M;\mathbb{Z})$ and $R_i\in \mathcal{D}_{n+2-k}(M)$. The following lemma then allows us to compare the masses $\mathbb{M}(T_i)$ to the real-homological widths ${\bf L}_{m,\mathbb{R}}(\overline{\xi})$.

\begin{lem}\label{fillinlem} Given $\delta>0$ and $L_1<\infty$, there exists $\eta(M,L_1,\delta)>0$ such that if $T_0,T_1,\ldots,T_r\in \mathcal{Z}_m(M;\mathbb{Z})$ is a sequence of integral $m$-cycles of mass 
\begin{equation}\label{fillinlem1}
\mathbb{M}(T_i)\leq L_1,
\end{equation}
with $T_0=T_r=0$, for which there exist $(m+1)$-currents of the form 
$$S_1,\ldots,S_r\in \mathcal{I}_{m+1}(M;\mathbb{Z})+\partial \mathcal{D}_{m+2}(M)$$ 
such that
\begin{equation}\label{fillinlem2}
\partial S_i=T_i-T_{i-1}\text{ and }\mathbb{M}(S_i)<\eta,
\end{equation}
then $\{T_i\}\in \mathscr{S}_{m,\delta}(M)$, with
\begin{equation}\label{realalmiso}
\Psi(\{T_i\})\equiv [\Sigma_{i=1}^rS_i]\text{ in }H_{m+1}(M;\mathbb{R}).
\end{equation}
\end{lem}
\begin{proof} To begin, we claim that there exists $\eta(M,L_1,\delta)>0$ such that $\mathbb{F}(T)<\delta$ for any integral cycle $T$ with $$\mathbb{M}(T)\leq 2L_1\text{ and }\|T\|_{(C^1)^*}<\eta.$$
Indeed, this is a simple consequence of the Federer-Fleming compactness theorem (\cite{Fed}, Theorem 4.2.17), since any sequence of integral cycles converging weakly to $0$ with uniformly bounded mass must also converge to $0$ in the flat norm. Applying this claim to the differences $T_i-T_{i-1}$ for a family of cycles $\{T_i\}$ satisfying (\ref{fillinlem1})-(\ref{fillinlem2}), we immediately deduce that $\{T_i\}\in \mathscr{S}_{m,\delta}(M)$ for $\eta(M,L_1,\delta)>0$ sufficiently small.

\hspace{3mm} To check (\ref{realalmiso}), fix (as in \cite{GMS1}, Sect. 5.4.1) a collection $\omega^1,\ldots,\omega^{b_{m+1}(M)}\in \mathcal{A}^{m+1}(M)$ of closed $(m+1)$-forms generating the integer lattice in $H^{m+1}_{dR}(M)$, and let
$$C(M):=\max_{1\leq i\leq b_{m+1}(M)}\|\omega^i\|_{L^{\infty}}.$$
Given $\{T_i\}\in \mathscr{S}_{m,\delta}(M)$ satisfying (\ref{fillinlem1})-(\ref{fillinlem2}), let $S_i'\in \mathcal{I}_{m+1}(M;\mathbb{Z})$ be a family of integer rectifiable $(m+1)$-currents satisfying 
\begin{equation}\label{intfillincond}
\partial S_i'=T_i-T_{i-1}\text{ and }\mathbb{M}(S_i')<\frac{\epsilon_0}{2}.
\end{equation}
For each $i=1,\ldots,r$, the difference
$$R_i:=S_i-S_i'$$
is then a cycle of the form $R_i\in \mathcal{Z}_{m+1}(M;\mathbb{Z})+\partial \mathcal{D}_{m+2}(M)$, and as a consequence, we see that
\begin{equation}\label{homintpairing}
\langle R_i,\omega\rangle\in \mathbb{Z}
\end{equation}
for every $\omega \in \mathcal{A}^{m+1}(M)$. In particular, (\ref{homintpairing}) holds for the generators $\omega^1,\ldots,\omega^{b_{m+1}(M)}$ chosen above. 

\hspace{3mm} On the other hand, the mass bounds in (\ref{fillinlem1}) and (\ref{intfillincond}) imply that
$$|\langle R_i,\omega\rangle| \leq C(M)(\eta+\frac{\epsilon_0}{2})$$
whenever $\|\omega\|_{L^{\infty}}\leq 1$. Thus, taking $\epsilon_0(M)$ and $\eta(M,L_1,\delta)>0$ small enough that
$$C(M)(\eta+\epsilon_0/2)<1,$$
it follows from (\ref{homintpairing}) that $\langle R_i,\omega^j\rangle=0$ for each $i=1,\ldots,r$ and $j=1,\ldots,b_{m+1}(M)$. Summing over $i=1,\ldots,r$, we therefore have
$$\langle \Sigma_{i=1}^r S_i',\omega^j\rangle=\langle \Sigma_{i=1}^rS_i,\omega^j\rangle$$
for each $j=1,\ldots, b_{m+1}(M)$, and (\ref{realalmiso}) follows.
\end{proof}

\subsection{A Decomposition Lemma for $S_{\alpha}(u,v)$.}\label{homints}\hspace{30mm}

\hspace{3mm} In this section, we prove that for weakly close maps $u,v\in \mathcal{E}^p(M,N)$, the current $S_{\alpha}(u,v)\in \mathcal{D}_{n+1-k}(M)$ of Lemma \ref{sdecomplem} admits a decomposition of the form
$$S_{\alpha}(u,v):=\Gamma+\partial R,$$
where $R\in \mathcal{D}_{n+2-k}(M)$, and $\Gamma\in \mathcal{I}_{n+1-k}(M)$ is integer rectifiable. 

\begin{lem}\label{hilem} For $p\in (k-1,k)$ and $L_2<\infty$, there exists $\epsilon(M,N,L_2,p)>0$ such that if $u,v\in \mathcal{E}^p(M,N)$ satisfy
\begin{equation}\label{hilemenerghyp}
E_p(u)+E_p(v)\leq L_2
\end{equation}
and
\begin{equation}\label{hilimlphyp}
\|u-v\|_{L^p}<\epsilon,
\end{equation}
then there exist $\Gamma\in \mathcal{I}_{n+1-k}(M)$ and $R\in \mathcal{D}_{n+2-k}(M)$ for which
\begin{equation}\label{intsdecomp}
S_{\alpha}(u,v)=\Gamma+\partial R.
\end{equation}
The same result holds if either $u$ or $v\in C^{\infty}(M,N)$.
\end{lem}
\begin{remark}\label{equivintsdecomp} By Lemma \ref{sdecomplem}, (\ref{intsdecomp}) is clearly equivalent to the statement that
$$S_{\alpha}(v)-S_{\alpha}(u)=\Gamma+\partial R'$$
for some $R'\in \mathcal{D}_{n+2-k}(M)$.
\end{remark}

\begin{proof} By the Fubini-type arguments of \cite{Wh1} and \cite{HangLin}, for maps $u,v \in \mathcal{E}^p(M,N)$ satisfying (\ref{hilemenerghyp}) and (\ref{hilimlphyp}), we can find a cubeulation $h:|K|\to M$ such that
\begin{equation}\label{skelenerbds}
E_p(u\circ h,|K^{k-1}|)+E_p(v\circ h,|K^{k-1}|)\leq C(M)L_2,
\end{equation}
\begin{equation}\label{skellpclose}
\|u\circ h-v\circ h\|_{L^p(|K^{k-1}|)}\leq C(M)\epsilon,
\end{equation}
where $K$ is a fixed cubical complex and the Lipschitz constants $Lip(h)$ and $Lip(h^{-1})$ are bounded independent of $u$ and $v$. 

\hspace{3mm} Moreover, since the singular sets $Sing(u)$ and $Sing(v)$ are contained in a finite union of $(n-k)$-dimensional submanifolds of $M$, we can choose this $h:|K|\to M$ such that the $(k-1)$-skeleton $h(|K^{k-1}|)$ lies a positive distance from $Sing(u)\cup Sing(v)$. Since $u$ and $v$ are Lipschitz away from $Sing(u)\cup Sing(v)$ by definition of $\mathcal{E}^p(M,N)$, it follows in particular that the restrictions $(u\circ h)|_{|K^{k-1}|}$ and $(v\circ h)|_{|K^{k-1}|}$ to the $(k-1)$-skeleton are again Lipschitz.

\hspace{3mm} Next (as in, e.g., \cite{Wh2}, Theorem 1.1), we note that the compactness of the embedding $W^{1,p}(|K^{k-1}|)\hookrightarrow C^0(|K^{k-1}|)$ implies the existence of a constant $\epsilon(M,N,L_2,p)>0$ such that (\ref{skelenerbds}) and (\ref{skellpclose}) imply
$$\|u\circ h-v\circ h\|_{C^0(|K^{k-1}|)}\leq \delta(N).$$
Here, $\delta(N)$ is chosen such that the $\delta(N)$-neighborhood $U_{\delta}$ of $N$ in $\mathbb{R}^L$ retracts $\pi_N:U_{\delta}\to N$ onto $N$. 

\hspace{3mm} Choosing such an $\epsilon$, we proceed to define a map $w:M\times [0,1]\to N$ satisfying $w(x,0)=u(x)$ and $w(x,1)=v(x)$ (compare \cite{Hang}, Lemma 2.2); throughout, we use the bi-Lipschitz map $h$ to identify $M$ and $|K|$, without comment. First, we set
$$w(x,0):=u(x)\text{ and }w(x,1)=v(x)\text{ for }x\in |K|,$$
and on $|K^{k-1}|\times [0,1]$, we define
$$w(x,t):=\pi_N(tv(x)+(1-t)u(x)).$$
For each $k$-cell $\sigma$ in $K$, the restriction
$$w|_{\partial(\sigma\times [0,1])}\in W^{1,p}(\partial (\sigma\times [0,1]),N)$$
of $w$ to $\partial (\sigma \times [0,1])$ is then a well-defined Sobolev map, satisfying an estimate of the form
$$\sup_{x\in \partial (\sigma\times[0,1])}\dist(x,\Sigma_0)\cdot |dw(x)|<\infty,$$
where $\Sigma_0=Sing(u)\times \{0\}\cup Sing(v)\times\{1\}$. Identifying $\sigma \times [0,1]$ with the $(k+1)$-ball $B_1^{k+1}$ in a bi-Lipschitz way, we can then extend $w$ to $\sigma \times [0,1]$ radially, setting $w(x)=\frac{x}{|x|}$. 

\hspace{3mm} We have now extended $w$ to a $W^{1,p}$ map on the whole $(k+1)$-skeleton $|\tilde{K}^{k+1}|$ of $\tilde{K}=K\otimes [0,1]$, satisfying 
$$\sup_{x\in |\tilde{K}^{k+1}|}\dist(x,\Sigma_1)\cdot |dw(x)|<\infty,$$
where $\Sigma_1$ is contained in a finite collection of Lipschitz curves in $|\tilde{K}^{k+1}|$. On each $(k+1)$-cell $\sigma$ of $K$, we can then extend $w$ from $\partial (\sigma \times [0,1])$ to $\sigma \times [0,1]$ as above, and carry on in this way, until finally we have the desired map
$$w\in W^{1,p}(M\times [0,1],N)$$
satisfying
$$w(x,0)=u(x)\text{ and }w(x,1)=v(x)\text{ in the trace sense},$$
and
\begin{equation}\label{wliplocest}
\sup_{x\in M}\dist(x,\Sigma)|dw(x)|<\infty,
\end{equation}
where $\Sigma\subset M\times [0,1]$ is a $(n+1-k)$-rectifiable set with $\mathcal{H}^{n+1-k}(\Sigma)<\infty$.

\hspace{3mm} Doubling this construction, we can evidently extend $w$ to a map 
$$w: M\times S^1\cong M\times [-3,3]/6\mathbb{Z}\to N$$
satisfying
\begin{equation}
w(x,t)=u(x)\text{ for }t\in [-2,-1],\text{ }w(x,t)=v(x)\text{ for }t\in [1,2],
\end{equation}
and, by virtue of (\ref{wliplocest}),
\begin{equation}
\|dw\|_{L^{k,\infty}(M)}<\infty.
\end{equation}
In particular, it follows from Remark \ref{wklkrk} that the homological singularity $T_{\alpha}(w)$ is an integral cycle in $M\times S^1$:
$$T_{\alpha}(w)=\partial S_{\alpha}(w)\in\mathcal{Z}_{n+1-k}(M\times S^1;\mathbb{Z}).$$

\hspace{3mm} Now, let $\pi:M\times S^1\to M$ be the obvious projection, and define currents $R\in \mathcal{D}_{n+2-k}(M)$ and $\Gamma\in \mathcal{I}_{n+1-k}(M;\mathbb{Z})$ to be the pushforwards
$$R:=\pi_{\#}(S_{\alpha}(w)\lfloor (M\times [-1.5,1.5]))$$
and
$$\Gamma:=\pi_{\#}(T_{\alpha}(w)\lfloor (M\times [-1.5,1.5])).$$
Fix a sequence $\psi_j\in C_c^{\infty}((-1.5,1.5))$ satisfying $0\leq \psi_j\leq 1$ and 
$$\psi_j\equiv 1\text{ on }[-1.5+\frac{1}{j},1.5-\frac{1}{j}];$$
since
$$S_{\alpha}(w)=S_{\alpha}(u)\times (-2,1)\text{ on }M\times (-2,-1)$$
and 
$$S_{\alpha}(w)=S_{\alpha}(v)\times (1,2)\text{ on }M\times (1,2),$$
it's clear that
\begin{equation}
R=\lim_{j\to\infty}R_j=\lim_{j\to\infty}\pi_{\#}(\psi_j(t)S_{\alpha}(w))
\end{equation}
and
\begin{equation}
\Gamma=\lim_{j\to\infty}\Gamma_j=\lim_{j\to\infty}\pi_{\#}(\psi_j(t)S_{\alpha}(w)).
\end{equation}

\hspace{3mm} For any $\zeta\in \Omega^{n+1-k}(M)$, we now compute
\begin{eqnarray*}
\langle \partial R_j,\zeta\rangle&=&\langle R_j,d\zeta\rangle\\
&=&\int_{M\times [-1.5,1.5]}\psi_j(t)w^*(\alpha)\wedge d(\pi^*\zeta)\\
&=&\int_{M\times [-1.5,1.5]}w^*(\alpha)\wedge d(\psi_j(t)\pi^*\zeta)\\
&&-\int_{M\times [-1.5,1.5]}w^*(\alpha)\wedge \psi_j'(t)dt\wedge \pi^*\zeta,
\end{eqnarray*}
which, by definition of $\Gamma_j$, gives
\begin{eqnarray*}
\langle \partial R_j-\Gamma_j,\zeta\rangle&=&(-1)^{n+2-k}\int_{M\times [-1.5,1.5]}w^*(\alpha)\wedge \pi^*\zeta\wedge \psi_j'(t)dt\\
&=&(-1)^{n+2-k}\int_{-1.5}^{-1}\psi_j'(t)dt\cdot \left(\int_Mu^*(\alpha)\wedge \zeta\right)\\
&&+(-1)^{n+2-k}\int_1^{1.5}\psi_j'(t)dt\left(\int_M v^*(\alpha)\wedge \zeta\right)dt\\
&=&(-1)^{n+2-k}\langle S_{\alpha}(u)-S_{\alpha}(v),\zeta\rangle
\end{eqnarray*}
for $j$ sufficiently large. 

\hspace{3mm} Passing to the limit $j\to\infty$, we conclude that
$$\partial R-\Gamma=(-1)^{n+2-k}(S_{\alpha}(u)-S_{\alpha}(v)),$$
and in particular,
$$S_{\alpha}(v)-S_{\alpha}(u)\in \mathcal{I}_{n+1-k}(M;\mathbb{Z})+\partial\mathcal{D}_{n+2-k}(M).$$
Recalling from Lemma \ref{sdecomplem} that the current $S_{\alpha}(u,v)$ differs from $S_{\alpha}(v)-S_{\alpha}(u)$ by the boundary of an $(n+2-k)$-current, it follows that
$$S_{\alpha}(u,v)\in \mathcal{I}_{n+1-k}(M;\mathbb{Z})+\partial \mathcal{D}_{n+2-k}(M)$$
as well, as desired.
\end{proof}

\subsection{Proof of Theorem \ref{lbdsthm}}\label{lbdspf}\hspace{30mm}

\hspace{3mm} Now, let $u,v\in C^{\infty}(M,N)$ be $(k-2)$-homotopic, and suppose there exists $\alpha\in \mathcal{A}^{k-1}(N)$ such that 
$$[u^*(\alpha)]-[v^*(\alpha)]\neq 0\in H_{dR}^{k-1}(M),$$
or, equivalently,
$$[S_{\alpha}(v)-S_{\alpha}(u)]\neq 0\in H_{n+1-k}(M;\mathbb{R}).$$
Evidently, such an $\alpha$ exists if and only if $u$ and $v$ induce different maps on the de Rham cohomology $H_{dR}^{k-1}(N)\to H_{dR}^{k-1}(M)$. 

\hspace{3mm} For every $\delta>0$, we define $\mathcal{C}^p_{\delta}(u,v)$ to be the collection of all finite sequences 
$$u_0,u_1,\ldots,u_r\in W^{1,p}(M,N)$$
such that $u_0=u$, $u_r=v$, and 
$$\|u_i-u_{i-1}\|_{L^p(M)}<\delta$$
for every $i=1,\ldots,r$. We then define
$$\gamma_p^{\delta}(u,v):=\inf\{\max_{0\leq j\leq r}E_p(u_j)\mid \{u_j\}_{j=1}^r\in \mathcal{C}^p_{\delta}(u,v)\},$$
and set
\begin{equation}\label{gammastardef}
\gamma_p^*(u,v):=\lim_{\delta\to 0}\gamma_p^{\delta}(u,v)=\sup_{\delta>0}\gamma_p^{\delta}(u,v).
\end{equation}

\begin{remark}\label{upperbdrk}
As observed in the introduction, it's clear from the definitions that $\gamma_p^*(u,v)\leq \gamma_p(u,v)$, since for any path $u_t\in W^{1,p}(M,N)$ from $u$ to $v$ and any $\delta>0$, we can find a finite sequence $0=t_0,t_1,\ldots,t_r=1$ such that $\{u_{t_i}\}\in \mathcal{C}^p_{\delta}(u,v)$. In particular, it follows from Theorem \ref{ubdsthm} that
$$\sup_{p<k}(k-p)\gamma_p^*(u,v)<\infty.$$
\end{remark}

We recall now the statement of Theorem \ref{lbdsthm}.

\begin{thm}\label{lbdsbody} For closed, oriented Riemannian manifolds $M^n,N$, maps $u,v\in C^{\infty}(M,N)$, and a $(k-1)$-form $\alpha\in \mathcal{A}^{k-1}(N)$ as above, such that
$$\overline{\xi}:=[S_{\alpha}(v)-S_{\alpha}(u)]\neq 0\in H_{n+1-k}(M;\mathbb{R}),$$
define
$$\Lambda(u,v):=\liminf_{p\to k}(k-p)\gamma_p^*(u,v).$$ 
Then we have the lower bound
\begin{equation}\label{mainthmbodbd}
\sigma_{k-1}{\bf L}_{n-k,\mathbb{R}}(\overline{\xi})\leq \lambda(\alpha)^{\frac{k}{k-1}}\Lambda(u,v).
\end{equation}
\end{thm}

\hspace{3mm} Most of the work in the proof of Theorem \ref{lbdsbody} is contained in the following lemma, which combines the results of Section \ref{limthms} and Lemma \ref{hilem}.

\begin{lem}\label{mainlem} For any $\eta\in (0,1)$, there exists $q(\eta)\in (k-1,k)$ with the property that for every $p\in (q,k)$, there exists $\delta_1(p)>0$ such that for any $\{u_i\}_{i=0}^r \in \mathcal{C}_{\delta_1}^p(u,v)$ satisfying
\begin{equation}\label{mapseqenerbd}
(k-p)\max_jE_p(u_j)\leq \Lambda(u,v)+\eta,
\end{equation}
we can find cycles $T_0,T_1,\ldots, T_r\in \mathcal{Z}_{n-k}(M;\mathbb{Z})$ with $T_0=T_r=0$ for which
$$\max_{0\leq i\leq r}\sigma_{k-1}\mathbb{M}(T_i)\leq\lambda(\alpha)^{\frac{k}{k-1}}(\Lambda(u,v)+\eta),$$
and currents
$$S_1,\ldots,S_r\in \mathcal{I}_{n+1-k}(M;\mathbb{Z})+\partial \mathcal{D}_{n+2-k}(M)$$
such that 
$$\partial S_i=T_i-T_{i-1},$$
\begin{equation}\label{mainlemfillinest}
\mathbb{M}(S_i)<3\eta,
\end{equation}
and
$$[\Sigma_{i=1}^rS_i]=[S_{\alpha}(v)-S_{\alpha}(u)]\in H_{n+1-k}(M;\mathbb{R}).$$
\end{lem}

\begin{proof} First, we appeal to Theorem \ref{bigcpctthm} to guarantee the existence of $q_0(\eta)=q_0(\Lambda(u,v)+\eta,\eta)>0$ such that for any $p\in (q_0,k)$ and any sequence of maps $u_1,\ldots,u_{r-1}\in W^{1,p}(M,N)$ satisfying (\ref{mapseqenerbd}), there exists a corresponding sequence $\tilde{u}_1,\ldots,\tilde{u}_{r-1}\in \mathcal{E}^p(M,N)$ satisfying
\begin{equation}\label{cpctcons1}
E_p(\tilde{u}_i)\leq \frac{C}{(k-p)^2},
\end{equation}
\begin{equation}\label{cpctcons2}
\|u_i-\tilde{u}_i\|_{L^p}^p\leq C(k-p)^{3p-2},
\end{equation}
and a sequence of integral cycles $T_i\in \mathcal{Z}_{n-k}(M;\mathbb{Z})$ and integral $(n+1-k)$-currents $\Gamma_i\in \mathcal{I}_{n+1-k}(M;\mathbb{Z})$ such that
\begin{equation}\label{cpctcons3}
\sigma_{k-1}\mathbb{M}(T_i)\leq \lambda(\alpha)^{\frac{k}{k-1}}[\Lambda(u,v)+\eta]
\end{equation}
\begin{equation}\label{cpctcons4}
T_{\alpha}(\tilde{u}_i)-T_i=\partial \Gamma_i,
\end{equation}
and
\begin{equation}\label{cpctcons5}
\mathbb{M}(\Gamma_i)<\eta.
\end{equation}

\hspace{3mm} Now, consider a family $\{u_i\}_{i=0}^r\in \mathcal{C}^p_{\delta}(u,v)$ for $p\in (q_0,k)$ satisfying (\ref{mapseqenerbd}). For $1\leq i\leq r-1$, choose 
$$\tilde{u}_i\in \mathcal{E}^p(M,N),\text{ }T_i\in \mathcal{Z}_{n-k}(M;\mathbb{Z})\text{, and }\Gamma_i\in \mathcal{I}_{n+1-k}(M;\mathbb{Z})$$
satisfying (\ref{cpctcons1})-(\ref{cpctcons5}), and extend these sequences trivially by setting 
$$\tilde{u}_0=u,\text{ }\tilde{u}_r=v,\text{ }T_0=T_r=0,\text{ and }\Gamma_0=\Gamma_r=0.$$ 
Next, setting
$$S_i:=S_{\alpha}(\tilde{u}_{i-1},\tilde{u}_i)+\Gamma_{i-1}-\Gamma_i,$$
for $i=1,\ldots,r$, we see that
$$\partial S_i=T_i-T_{i-1},$$
and, by Lemma \ref{hilem}, these $S_i$ have the form
$$S_i\in \mathcal{I}_{n+1-k}(M;\mathbb{Z})+\partial\mathcal{D}_{n+2-k}(M).$$
Moreover, since 
$$\Sigma_{i=1}^r(\Gamma_{i-1}-\Gamma_i)=\Gamma_0-\Gamma_r=0,$$
and (by Lemma \ref{sdecomplem}) 
$$S_{\alpha}(\tilde{u}_i)-S_{\alpha}(\tilde{u}_{i-1})-S_{\alpha}(\tilde{u}_{i-1},\tilde{u}_i)\in \partial D_{n+2-k}(M),$$
it follows that 
\begin{equation}
[\Sigma_{i=1}^rS_i]=[S_{\alpha}(v)-S_{\alpha}(u)]\in H_{n+1-k}(M;\mathbb{R}).
\end{equation}

\hspace{3mm} To complete the proof of the lemma, it remains to establish the mass bound (\ref{mainlemfillinest}) for these $S_i$, for $p>q(\eta)$ sufficiently large and $\delta=\delta_1(p)$ sufficiently small. To estimate the mass of 
$$S_i=S_{\alpha}(\tilde{u}_{i-1},\tilde{u}_i)+\Gamma_{i-1}-\Gamma_i,$$
we first apply (\ref{cpctcons5}) to see that
$$\mathbb{M}(S_i)\leq \mathbb{M}(S_{\alpha}(\tilde{u}_{i-1},\tilde{u}_i))+2\eta.$$
Now, by Lemma \ref{sdecomplem} and the energy bound (\ref{cpctcons1}), we know that
\begin{eqnarray*}
\mathbb{M}(S_{\alpha}(\tilde{u}_{i-1},\tilde{u}_i))&\leq &C [E_p(\tilde{u}_{i-1})^{\frac{k-1}{p}}+E_p(\tilde{u}_i)^{\frac{k-1}{p}}]\|\tilde{u}_i-\tilde{u}_{i-1}\|_{L^p}^{1+p-k}\\
&\leq &C (k-p)^{-2(k-1)/p}\|\tilde{u}_i-\tilde{u}_{i-1}\|_{L^p}^{1+p-k}.
\end{eqnarray*}
Moreover, by (\ref{cpctcons2}), we have
$$\|u_i-\tilde{u}_i\|_{L^p}^p\leq C(k-p)^{3p-2},$$
while by definition of $\mathcal{C}^p_{\delta}(u,v)$, we have also
$$\|u_i-u_{i-1}\|_{L^p}^p<\delta^p.$$
Taking $\delta=\delta_1(p)=(k-p)^{3-2/p}$ and combining the estimates above, it follows in particular that
$$\|\tilde{u}_i-\tilde{u}_{i-1}\|_{L^p}^p\leq C'(k-p)^{3p-2},$$
and consequently,
\begin{eqnarray*}
\mathbb{M}(S_{\alpha}(\tilde{u}_{i-1},\tilde{u}_i))&\leq &C(k-p)^{-2(k-1)/p}\|\tilde{u}_i-\tilde{u}_{i-1}\|_{L^p}^{1+p-k}\\
&\leq & C'(k-p)^{-\frac{2(k-1)}{p}+(3p-2)\frac{1+p-k}{p}}\\
&=&C'(k-p)^{1-3(k-p)}.
\end{eqnarray*}
Since $\lim_{p\to k}(k-p)^{1-3(k-p)}=0$, we can therefore choose $q(\eta)\in (k-1,k)$ such that
$$\mathbb{M}(S_{\alpha}(\tilde{u}_{i-1},\tilde{u}_i))<\eta,$$
and, consequently,
$$\mathbb{M}(S_i)<3\eta$$
for $p\in (q,k)$. This completes the proof.
\end{proof}

\hspace{3mm} Combining the preceding lemma with the results of Lemma \ref{fillinlem}, we complete the proof of Theorem \ref{lbdsbody} as follows.

\begin{proof} Fix some $\delta\in (0,1)$. By Lemma \ref{fillinlem}, we can find $\eta(\delta)\in (0,\delta)$ such that for any sequence $\{T_i\}_{i=0}^r\subset\mathcal{Z}_{n-k}(M;\mathbb{Z})$ satisfying $T_0=T_r=0,$
\begin{equation}\label{cycleseq1}
\mathbb{M}(T_i)\leq C(\alpha)[\Lambda(u,v)+1]
\end{equation}
and
\begin{equation}
T_i-T_{i-1}=\partial S_i
\end{equation}
for some $S_1,\ldots,S_r\in \mathcal{I}_{n+1-k}(M;\mathbb{Z})+\partial D_{n-k+2}(M)$ with
\begin{equation}\label{cycleseq3}
\mathbb{M}(S_i)<3\eta,
\end{equation}
we have $\{T_i\}\in \mathscr{S}_{n-k,\delta}(M)$, with the associated homology class $\Psi(\{T_i\})$ satisfying
\begin{equation}\label{cyclesalmiso}
\Psi(\{T_i\})\equiv [\Sigma_{i=1}^rS_i]\text{ in }H_{n-k+1}(M;\mathbb{R}).
\end{equation}

\hspace{3mm} Now, with $\eta(\delta)\in (0,\delta)$ as above, let $q(\eta)$ be as in Lemma \ref{mainlem}, and choose $p\in (q,k)$ such that
$$(k-p)\gamma_p^*(u,v)\leq \Lambda(u,v)+\eta/2.$$
Then, choose $\delta_1(p)>0$ according to Lemma \ref{mainlem}, and select some family $\{u_i\}_{i=0}^r\in \mathcal{C}_{\delta_1}^p(u,v)$ such that
\begin{eqnarray*}
(k-p)\max_iE_p(u_i)&\leq & (k-p)\gamma_p^{\delta_1}(u,v)+\eta/2\\
&\leq & \Lambda(u,v)+\eta.
\end{eqnarray*}
By Lemma \ref{mainlem}, we can associate to these $\{u_i\}_{i=0}^r$ a family of cycles $T_0,\ldots,T_r\in \mathcal{Z}_{n-k}(M;\mathbb{Z})$ and currents $S_1,\ldots,S_r\in \mathcal{I}_{n+1-k}(M;\mathbb{Z})+\partial \mathcal{D}_{n-k+2}(M)$ satisfying (\ref{cycleseq1})-(\ref{cycleseq3}), as well as the sharper mass bound
\begin{eqnarray*}
\sigma_{k-1}\max_i\mathbb{M}(T_i)&<&\lambda(\alpha)^{\frac{k}{k-1}}\Lambda(u,v)+C(\alpha)\eta\\
&<&\lambda(\alpha)^{\frac{k}{k-1}}\Lambda(u,v)+C(\alpha)\delta,
\end{eqnarray*}
and the homological condition
$$[\Sigma_{i=1}^rS_i]\equiv [S_{\alpha}(v)-S_{\alpha}(u)]\text{ in }H_{n-k+1}(M;\mathbb{R}).$$

\hspace{3mm} In particular, it follows that there exists $\{T_i\}\in \mathscr{S}_{n-k,\delta}(M)$ satisfying 
$$\Psi(\{T_i\})\equiv [S_{\alpha}(v)-S_{\alpha}(u)]\in H_{n+1-k}(M;\mathbb{R})$$
and
$$\max_i\sigma_{k-1}\mathbb{M}(T_i)<\lambda(\alpha)^{\frac{k}{k-1}}\Lambda(u,v)+C(\alpha)\delta.$$
Recalling the notation of Section \ref{almsubsec}, this means precisely that
$$\sigma_{k-1}{\bf L}_{n-k,\mathbb{R},\delta}([S_{\alpha}(v)-S_{\alpha}(u)])<\lambda(\alpha)^{\frac{k}{k-1}}\Lambda(u,v)+C(\alpha)\delta.$$
Finally, taking $\delta\to 0$, we arrive at the desired inequality
$$\sigma_{k-1}{\bf L}_{n-k,\mathbb{R}}([S_{\alpha}(v)-S_{\alpha}(u)])\leq \lambda(\alpha)^{\frac{k}{k-1}}\Lambda(u,v).$$

\end{proof}

\section{Associated $p$-Harmonic Maps}\label{pharms}

\hspace{3mm} In this short section, we demonstrate the existence of mountain pass critical points for the $p$-energy functional associated with energy lying between $\gamma_p^*(u,v)$ and $\gamma_p(u,v)$.

\begin{thm}\label{pharmthm} For any $(k-2)$-homotopic maps $u,v\in C^{\infty}(M,N)$ and $p\in (1,k)\setminus \mathbb{N}$ such that 
\begin{equation}
\max\{E_p(u),E_p(v)\}<\gamma_p^*(u,v),
\end{equation}
there exists a stationary $p$-harmonic map $w\in W^{1,p}(M,N)$ such that
$$\gamma_p^*(u,v)\leq E_p(w)\leq \gamma_p(u,v).$$
\end{thm}

\hspace{3mm} We produce these $p$-harmonic maps by applying standard mountain pass methods to the generalized Ginzburg-Landau functionals studied by Wang in \cite{Wang}. Fixing once again an isometric embedding 
$$N\subset \mathbb{R}^L$$
of our target manifold $N$ into some higher-dimensional Euclidean space, we consider a function $F: C^{\infty}(\mathbb{R}^L)$ satisfying
$$F(x)=\dist(x,N)^2\text{ when }\dist(x,N)<\delta_N$$
on the $\delta_N$-tubular neighborhood of $N$,
$$F(x)\geq \delta_N^2\text{ when }\dist(x,N)\geq \delta_N,$$
and (for technical reasons)
$$F(x)=|x|\text{ for }|x|>R_0,\text{ some large radius}.$$
For $p\in (1,\infty)$ and $\epsilon>0$, the generalized Ginzburg-Landau functionals
$$E_{p,\epsilon}: W^{1,p}(M,\mathbb{R}^L)\to \mathbb{R}$$
can then be defined by
\begin{equation}
E_{p,\epsilon}(w):=\int_M(|dw|^p+\epsilon^{-p}F(w)).
\end{equation}

\hspace{3mm} For the $(k-2)$-homotopic maps $u,v\in C^{\infty}(M,N)$ and $p\in (1,k)$, we then define the mountain pass energies $\gamma_{GL,p,\epsilon}(u,v)$ to be the infimum
\begin{equation}
\gamma_{GL,p,\epsilon}(u,v):=\inf\{\max_{t\in [0,1]} E_{p,\epsilon}(u_t)\mid u_0=u,\text{ }u_1=v\}
\end{equation}
of the maximum energy $\max_{t\in [0,1]}E_{p,\epsilon}(u_t)$ over all continuous paths $t\mapsto u_t$ in $C^0([0,1],W^{1,p}(M,\mathbb{R}^L))$ from $u_0=u$ to $u_1=v$. It follows immediately that
\begin{equation}\label{gleupper}
\gamma_{GL,p,\epsilon}(u,v)\leq \gamma_p(u,v)
\end{equation}
for every $\epsilon>0$, since any continuous family $u_t\in W^{1,p}(M,N)$ connecting $u$ to $v$ through $N$-valued maps satisfies
$$\gamma_{GL,p,\epsilon}(u,v)\leq \max_{t\in [0,1]}E_{p,\epsilon}(u_t)=\max_{t\in [0,1]}E_p(u_t).$$

\hspace{3mm} Now, since the generalized Ginzburg-Landau energies $E_{p,\epsilon}$ are $C^1$ functionals on the Banach space $W^{1,p}(M,\mathbb{R}^L)$, and satisfy a Palais-Smale condition (see, e.g., \cite{St18}, Section 7.1), we can appeal to standard existence results for critical points of mountain pass type (see \cite{Gh}, Chapter 6) to arrive at the following lemma.

\begin{lem}\label{glmpsols}
For any $p\in (1,k)$ and $\epsilon>0$, if 
\begin{equation}
\gamma_{GL,p,\epsilon}(u,v)>\max\{E_{p,\epsilon}(u),E_{p,\epsilon}(v)\}=\max\{E_p(u),E_p(v)\},
\end{equation}
then there exists a critical point $w_{\epsilon}$ of $E_{p,\epsilon}$ of energy
$$E_{p,\epsilon}(w_{\epsilon})=\gamma_{GL,p,\epsilon}(u,v).$$
\end{lem}

\hspace{3mm} In light of the upper bound (\ref{gleupper}) for the energies $\gamma_{GL,p,\epsilon}(u,v)$, the critical points $w_{\epsilon}$ given by Lemma \ref{glmpsols} have uniformly bounded energies $E_{p,\epsilon}(w_{\epsilon})$ as $\epsilon\to 0$. For non-integer $p\in (1,k)\setminus \mathbb{N}$, it therefore follows from the compactness results of \cite{Wang} (namely, \cite{Wang}, Corollary B) that some subsequence $w_{\epsilon_j}$ converges strongly to a stationary $p$-harmonic map $w\in W^{1,p}(M,N)$. In particular, we have the following existence result.

\begin{prop}\label{pharmprop}
For every $p\in (1,k)\setminus \mathbb{N}$, if
$$\max\{E_p(u),E_p(v)\}<\lim_{\epsilon\to 0}\gamma_{GL,p,\epsilon}(u,v)\left(=\sup_{\epsilon>0}\gamma_{GL,p,\epsilon}(u,v)\right),$$
then there exists a stationary $p$-harmonic map $w\in W^{1,p}(M,N)$ of energy
\begin{equation}
E_p(w)=\sup_{\epsilon>0}\gamma_{GL,p,\epsilon}(u,v).
\end{equation}
\end{prop}

\hspace{3mm} From (\ref{gleupper}), it's clear that the maps $w$ obtained in Proposition \ref{pharmprop} satisfy $$E_p(w)\leq \gamma_p(u,v).$$
Thus, to complete the proof of Theorem \ref{pharmthm}, it remains only to establish the lower bound
\begin{equation}\label{gllowbds}
\gamma_p^*(u,v)\leq \sup_{\epsilon>0}\gamma_{GL,p,\epsilon}(u,v).
\end{equation}
This will follow fairly easily from the definition of $\gamma_p^*(u,v)$ and the following easy lemma.

\begin{lem}\label{dcomplem} For every $\eta>0$, there exists some $\epsilon_0(p,\eta)>0$ such that if $\epsilon<\epsilon_0$ and $w\in W^{1,p}(M,\mathbb{R}^L)$ satisfies
$$E_{p,\epsilon}(w)<\gamma_p(u,v)+1,$$
then there exists $w'\in W^{1,p}(M,N)$ such that 
$$\|w-w'\|_{L^p}<\eta$$
and
$$E_p(w')\leq E_{p,\epsilon}(w)+\eta.$$
\end{lem}

\begin{proof} This is another simple proof by contradiction via compactness. If, for some $\eta>0$, no such $\epsilon_0(p,\eta)$ existed, then we could find a sequence $\epsilon_j\to 0$ and $w_j\in W^{1,p}(M,\mathbb{R}^L)$ such that
$$\limsup_{j\to\infty}E_{p,\epsilon_j}(w_j)\leq \gamma_p(u,v)+1<\infty$$
and for every $j$ and $w'\in W^{1,p}(M,N)$, either
\begin{equation}\label{contracond}
\|w_j-w'\|_{L^p}>\eta\text{ or }E_p(w')>E_{p,\epsilon}(w_j)+\eta.
\end{equation}
But, passing to a further subsequence, we can find $w\in W^{1,p}(M,\mathbb{R}^L)$ for which $w_j\rightharpoonup w$ in $W^{1,p}$ and 
\begin{equation}\label{lpcontra}
\|w_j-w\|_{L^p}\to 0.
\end{equation}
Since the energies $E_{p,\epsilon_j}(w_j)$ are uniformly bounded as $\epsilon_j\to 0$, we see that
$$\lim_{j\to\infty}\int_M F(w_j)=0,$$
and consequently $w\in W^{1,p}(M,N)$. And of course, it follows from the weak convergence that
\begin{equation}
E_p(w)\leq \liminf_{j\to\infty}E_{p,\epsilon}(w_j),
\end{equation}
which, together with (\ref{lpcontra}), contradicts (\ref{contracond}).
\end{proof}

\hspace{3mm} Now, for any $\delta>0$, choose $\epsilon_0=\epsilon_0(p,\delta/3)$ according to Lemma \ref{dcomplem}, and for $\epsilon<\epsilon_0$, consider a path $w_t\in W^{1,p}(M,\mathbb{R}^L)$ connecting $w_0=u$ to $w_1=v$, such that
$$\max_{t\in [0,1]}E_{p,\epsilon}(w_t)\leq \gamma_{GL,p,\epsilon}(u,v)+\delta.$$
Select a sequence of times $0=t_0<t_1<\cdots<t_r=1$ such that 
\begin{equation}\label{wilpdiff}
\|w_{t_i}-w_{t_{i-1}}\|_{L^p}<\frac{\delta}{3},
\end{equation}
and for each $1\leq i\leq r-1$, appeal to Lemma \ref{dcomplem} to find a map $u_i\in W^{1,p}(M,N)$ such that
\begin{equation}\label{uiwilpdiff}
\|u_i-w_{t_i}\|_{L^p}<\frac{\delta}{3}
\end{equation}
and
\begin{equation}\label{uiglbd}
E_p(u_i)\leq \gamma_{GL,p,\epsilon}(u,v)+2\delta.
\end{equation}
It follows from (\ref{wilpdiff}) and (\ref{uiwilpdiff}) that the sequence
$$u=u_0,u_1,\ldots,u_{r-1},u_r=v$$
belongs to $\mathcal{C}_{\delta}^p(u,v)$, and from (\ref{uiglbd}), we therefore see that
\begin{equation}
\gamma_p^{\delta}(u,v)\leq \gamma_{GL,p,\epsilon}(u,v)+2\delta.
\end{equation}
In particular, we've now shown that
$$\gamma_p^{\delta}(u,v)\leq \sup_{\epsilon>0}\gamma_{GL,p,\epsilon}(u,v)+2\delta$$
for every $\delta>0$. Taking the limit as $\delta\to 0$, we arrive at the desired lower bound (\ref{gllowbds}); together with (\ref{gleupper}) and Proposition \ref{pharmprop}, this completes the proof of Theorem \ref{pharmthm}.

\section{Appendix}
\subsection{Remarks on the Proof of Lemma \ref{slicelemma}}\label{slicelemapp}\hspace{30mm}

\hspace{3mm} Here, we provide a few comments on how Lemma \ref{slicelemma} follows from the arguments of (\cite{HangLin}, Section 3 and 4; see also \cite{Hang}, Section 2).

\hspace{3mm} To begin, we fix some (piecewise) smooth cubeulation $h:|K|\to M$ of $M$ (following, for instance, the construction in \cite{Wh1}), where $K$ is a cubical complex all of whose faces are isometric to $[-1,1]^n$, and 
$$Lip(h)+Lip(h^{-1})\leq C(M).$$
We remark that it is enough to prove Lemma \ref{slicelemma} for rational scales $\delta=\frac{1}{m}$, and henceforth (as in \cite{Hang}) we restrict ourselves to this case. Beginning from the initial cubeulation $K=K_1$ above, we can then subdivide each $n$-cell into $m^n$ copies of $[-\delta,\delta]^n$, to obtain a new complex $K_{\delta}$ with the same underlying space $|K_{\delta}|=|K|$ as the initial one.

\hspace{3mm} Now, consider an isometric embedding $M\subset \mathbb{R}^L$ into a high-dimensional Euclidean space, and fix $\epsilon(M)>0$ such that the nearest point projection $\Pi_M$ onto $M$ is well-defined and smooth on the $\epsilon(M)$-neighborhood of $M$ in $\mathbb{R}^L$. As in \cite{HangLin}, define the map
$$H: |K|\times B_{\epsilon}^L(0)\to M$$
by setting
$$H(x,\xi):=\Pi_M(h(x)+\xi).$$
By choosing $\epsilon(M)$ sufficiently small, we can then arrange that
\begin{equation}\label{obvlipest}
\|H\|_{Lip}\leq C(M),
\end{equation}
and the maps 
$$h_{\xi}:=H(\cdot,\xi):|K|\to M$$
are invertible, with
\begin{equation}\label{hxibilip}
Lip(h_{\xi})+Lip(h_{\xi}^{-1})\leq C(M).
\end{equation}
Moreover, we can arrange that the Jacobian determinant
$$JH_{j,\delta}(x,\xi):=\det(DH_{j,\delta}(x,\xi)\circ [DH_{j,\delta}(x,\xi)]^*)^{1/2}$$
of the restriction $H_{j,\delta}:=H|_{|K^j_{\delta}|}$ of $H$ to the $j$-skeleton $K^j_{\delta}$ has a uniform lower bound
\begin{equation}\label{jaclbd}
JH_{j,\delta}(x,\xi)\geq C(M)^{-1}>0.
\end{equation}

\hspace{3mm} Next, as in Section 4 of \cite{HangLin}, fix $y\in M$, and consider the map
$$\psi:|K|\times\{\xi \in \mathbb{R}^L\mid \xi\perp T_yM,\text{ }|\xi|\leq \epsilon(M)\}\to |K|\times \mathbb{R}^L$$
mapping the product of $|K|$ with the normal disk $D_{\epsilon}^{\perp}(y)$ to $M$ at $y$ to $|K|\times \mathbb{R}^L$ by
$$\psi(x,y):=(x,y+\xi-h(x)).$$
For any subset $A\subset |K|$, we then observe that
$$H^{-1}(y)\cap A\subset \psi(A).$$
In particular, for the skeleta $|K^j_{\delta}|$ of $K_{\delta}$, it follows that
\begin{equation}\label{preimest}
\mathcal{H}^{L-n+j}(H^{-1}(y)\cap |K^j_{\delta}|)\leq \mathcal{H}^{L-n+j}(\psi(|K^j_{\delta}|\times D_{\epsilon}^{\perp}(y))\leq C(M)\delta^{j-n},
\end{equation}
where in the last inequality we have used the area formula for the map $\psi$ together with the simple estimate $\mathcal{H}^j(|K^j_{\delta}|)\leq C(K_1)\delta^{j-n}$ (since $K^j_{\delta}$ comprises $C(K_1)\delta^{-n}$ $j$-cells of size $\delta$).

\hspace{3mm} Armed with the estimates (\ref{obvlipest})-(\ref{preimest}), one can now employ the coarea formula and argue exactly as in Section 3 of \cite{HangLin} to conclude the proof of Lemma \ref{slicelemma}.

\subsection{Upper Bounds for $\gamma_p(u,v)$ from the Hang-Lin Construction}\label{upperbds}\hspace{30mm}

\hspace{3mm} We recall now the construction of \cite{HangLin} (cf. also \cite{BL} in the case that either $u$ or $v$ is constant), and explain how it leads immediately to a proof of Theorem \ref{ubdsthm}

\begin{prop}\label{goodpath}\emph{(cf. \cite{HangLin})} Let $u,v\in C^{\infty}(M,N)$ be $(k-2)$-homotopic for some $k\leq n=dim(M)$. Then there is a path of maps $t\mapsto u_t$ with $u_0=u$, $u_1=v$, continuous in $W^{1,p}(M,N)$ for every $1\leq p<k$, such that
\begin{equation}
\sup_{t\in [0,1]}E_p(u_t)\leq \frac{C}{k-p}
\end{equation}
for some $C$ independent of $p$.
\end{prop}

\begin{proof} To begin, fix a smooth cubeulation $h:|K|\to M$, where $K$ is a cubical complex built of $n$-cells isometric to $[-1,1]^n$. In what follows, we will frequently identify $M$ and $|K|$ without comment. Since we've taken $u$ and $v$ to be smooth, note that the restrictions $u|_{|K^j|}$ and $v_{|K^j|}$ of $u$ and $v$ to the lower-dimensional skeleta of $K$ define Lipschitz maps from $|K^j|$ to $N$.

\hspace{3mm} Recalling the terminology of Section \ref{slices}, we observe now that there exists a path of maps $u_t$ connecting $u$ to $u\circ \Phi_k$, such that $t\mapsto u_t$ is continuous in $W^{1,p}(M,N)$ for each $p<k$, with the desired energy bounds. Indeed, it follows directly from Lemma \ref{philem} that the path
$$[0,1]\ni s\mapsto u\circ \phi_{n,s}$$
connecting $u\circ \Phi_n$ to $u$ satisfies the desired properties, as do the paths
$$[0,1]\ni s\mapsto u\circ \phi_{j,s}\circ\Phi_{j+1}$$
connecting $u\circ \Phi_{j+1}$ to $u\circ \Phi_j$ for each $k\leq j\leq n-1$. Concatenation yields the desired path $u_t$ from $u$ to $u\circ \Phi_k$, and in the same way we can construct such a path connecting $v$ to $v\circ \Phi_k$. 

\hspace{3mm} It remains now to construct a path of maps $u_t$ from $M$ to $N$ connecting $u\circ\Phi_k$ to $v\circ \Phi_k$, in such a way that $t\mapsto u_t$ is continuous in $W^{1,p}(M,N)$ and 
$$\max_t E_p(u_t)\leq \frac{C}{k-p}$$
for every $1\leq p<k$. In fact, it is enough to construct such a path of maps
$$w_t:|K^k|\to N$$
connecting $u\circ \phi_{k,0}$ to $v\circ \phi_{k,0}$ on the $k$-skeleton $|K^k|$, since we can then take $u_t:=w_t\circ \Phi_{k+1}$ to obtain the desired path of maps on $M$. In the remainder of the proof, we construct such a path $w_t:|K^k|\to N$.

\hspace{3mm} Since the maps $u$ and $v$ are $(k-2)$-homotopic, their restrictions $u|_{|K^{k-2}|}$, $v|_{|K^{k-2}|}$ to the $(k-2)$-skeleton $|K^{k-2}|$ are homotopic, by definition. And since the pair $(|K^{k-1}|,|K^{k-2}|)$ satisfies the homotopy extension property (cf. \cite{HangLin}, Proposition 2.1), we can therefore find a map 
$$u_2:|K^{k-1}|\to N$$
such that $u|_{|K^{k-1}|}$ is homotopic to $u_2$ on $|K^{k-1}|$ and $u_2$ agrees with $v$ 
$$u_2|_{|K^{k-2}|}=v|_{|K^{k-2}|}$$
on $|K^{k-2}|$. Moreover, it's easy to check (cf. \cite{HangLin}, Sections 2.2-2.3) that we can take both the map $u_2$ and the homotopy $f: |K^{k-1}|\times [0,1]\to N$ from $u|_{|K^{k-1}|}$ to $u_2$ to be Lipschitz. The precomposition $f_t\circ \Phi_k$ of the homotopy $f_t$ with $\Phi_k$ then gives us a path of maps $M\to N$ connecting $u\circ \Phi_k$ to $u_2\circ\Phi_k$, which evidently satisfies the desired estimates and continuity properties in $W^{1,p}(M,N)$ for $1\leq p<k$. 

\hspace{3mm} In particular, to complete the proof of the proposition, we now see that it suffices to construct a path of maps $w_t:|K^k|\to N$, continuous in $W^{1,p}(|K^k|,N)$ for $1\leq p<k$, satisfying
$$\max_t E_p(w_t,|K^k|)\leq \frac{C}{k-p},$$
that connects $v\circ \phi_{k,0}$ to $u_2\circ \phi_{k,0},$ where $u_2\in Lip(|K^{k-1}|,N)$ agrees with $v$ on the $(k-2)$-skeleton $|K^{k-2}|$. To do this, we enumerate the $(k-1)$-cells $\sigma_1,\ldots,\sigma_m\in K^{k-1}\setminus K^{k-2}$, and define maps $w_0,\ldots,w_m\in W^{1,p}(|K^k|,N)$ by
$$w_i:=f_i\circ \phi_{k,0},$$
where the maps $f_i\in Lip(|K^{k-1}|,N)$ are defined by $f_0=v|_{|K^{k-1}|},$ $f_m=u_2$, and
$$f_i(x):=v(x)\text{ for }x\in |K^{k-1}|\setminus (\sigma_1\cup\cdots\cup \sigma_i),$$
$$f_i(x):=u_2(x)\text{ for }x\in \sigma_1\cup\cdots\cup\sigma_i,$$
for $1\leq i\leq m$. (That these $f_i$ are Lipschitz follows from the fact that $u_2=v$ on $|K^{k-2}|$.) We claim that each $w_i$ can be deformed into $w_{i+1}$ through a path of maps $w_t$ satisfying the desired properties; once we have constructed these paths, concatenation evidently gives the desired path from $v\circ\phi_{k,0}$ to $u_2\circ\phi_{k,0}$.

\hspace{3mm} Now, fix $i\in \{1,\ldots,m\}$. By construction, the maps $f_{i-1},f_i\in Lip(|K^{k-1}|,N)$ coincide on the complement $|K^{k-1}|\setminus \sigma_i$ of the $(k-1)$-cell $\sigma_i$. Consider the star neighborhood 
$$V:=\bigcup\{\Delta \in K^k\text{ a }k\text{-cell}\mid \sigma_i\subset \partial \Delta\}$$
of $\sigma_i$, which we can identify in a bi-Lipschitz way with
$$W=\bigcup_{j=1}^aW_j\subset \mathbb{R}^{k-1+a},$$
where
$$W_j:=\{(x,0,\ldots,0)+te_{k-1+j}\mid x\in [-1,1]^{k-1},\text{ }0\leq t\leq 1\},$$
and $a$ is simply the number of distinct $k$-cells for which
$$\sigma\cong [-1,1]^{k-1}\times \{0\}\subset \mathbb{R}^{k-1+a}$$
is a face. 

\hspace{3mm} Next, note that the boundary
$$\partial V\cong W\cap \partial [-1,1]^{k-1+a}$$
lies in $|K^{k-1}|\setminus \sigma_i$, so that the maps
$$f_i=f_{i-1}=:g\in Lip(\partial V,N)$$
agree on $\partial V$. For $t\in [0,\frac{1}{2}]$, we can then define maps $w_{i-1+t}:|K^k|\to N$ by setting
$$w_{i-1+t}:=w_{i-1}=w_i\text{ on }|K^k|\setminus V,$$
and (identifying $V$ with $W$)
$$w_{i-1+t}(x):=w_{i-1}\left(\frac{x}{\max\{(1-2t),|x|_{\infty}\}}\right)\text{ for }x\in V.$$
We can then check by direct computation, as in the proof of Lemma \ref{philem}, that $[0,\frac{1}{2}]\ni t\mapsto w_{i-1+t}$ satisfies the desired energy estimates and continuity properties, while connecting $w_{i-1}$ to the map $w_{i-0.5}$ given by
$$w_{i-0.5}:=w_i\text{ on }|K^k|\setminus V$$
and
$$w_{i-0.5}(x):=g(x/|x|_{\infty})\text{ for }x\in V.$$
Since $w_i|_{\partial V}=g$ as well, we can employ the same construction to obtain a path 
$$[\frac{1}{2},1]\ni t\mapsto w_{i-1+t}$$
connecting $w_{i-0.5}$ to $w_i$ in the desired way. We have thereby constructed a path $[0,1]\ni t\mapsto w_{i-1+t}:|K^k|\to N$ from $w_{i-1}$ to $w_i$ satisfying the desired estimates, completing the proof.
\end{proof}

\end{document}